\documentclass[reqno]{amsart}
\usepackage{amsmath}
\usepackage{amsthm, amscd, amssymb, amsfonts, amsbsy}
\usepackage[usenames, dvipsnames]{color}
\usepackage{enumerate}
\usepackage{mathrsfs}
\usepackage{verbatim}

\numberwithin{equation}{section}

\theoremstyle{plain}
\newtheorem{theorem}{Theorem}[section]
\newtheorem{lemma}[theorem]{Lemma}

\theoremstyle{definition}
\newtheorem{definition}[theorem]{Definition}

\theoremstyle{remark}
\newtheorem{remark}[theorem]{Remark}
\newtheorem{notation}{Notation}[section]

\makeatletter
\def\dashint{\operatorname%
{\,\,\text{\bf--}\kern-.98em\DOTSI\intop\ilimits@\!\!}}
\makeatother

\def\bR{\mathbb{R}}
\def\cA{\mathcal{A}}

\def\cL{\mathcal{L}}
\def\cG{\mathcal{G}}
\def\cF{\mathcal{F}}
\def\cU{\mathcal{U}}

\begin{document}
\title[Stokes system]{Gradient estimates for Stokes systems with Dini mean oscillation coefficients}

\author[J. Choi]{Jongkeun Choi}
\address[J. Choi]{
School of Mathematics, Korea Institute for Advanced Study, 85 Hoegiro, Dongdaemun-gu, Seoul 02455, Republic of Korea}

\email{jkchoi@kias.re.kr}

\thanks{J. Choi was supported by Basic Science Research Program
through the National Research Foundation of Korea (NRF) funded by the Ministry of Education
(2017R1A6A3A03005168)}

\author[H. Dong]{Hongjie Dong}
\address[H. Dong]{Division of Applied Mathematics, Brown University, 182 George Street, Providence, RI 02912, USA}

\email{Hongjie\_Dong@brown.edu}

\thanks{H. Dong was partially supported by the NSF under agreement DMS-1600593.}


\begin{abstract}
We study the stationary Stokes system in divergence form.
The coefficients are assumed to be merely measurable in one direction and have Dini mean oscillations in the other directions.
We prove that if $(u,p)$ is a weak solution of the system, then $(Du,p)$ is bounded and its certain linear combinations are continuous.
We also prove a weak type-$(1,1)$ estimate for $(Du,p)$ under a stronger assumption on the $L^1$-mean oscillation of the coefficients.
The corresponding results up to the boundary on a half ball are also established. These results are new even for elliptic equations and systems.
\end{abstract}

\maketitle

\section{Introduction}		\label{Sec1}
We consider the stationary Stokes system with variable coefficients
\begin{equation}		\label{171219@eq5a}
\left\{
\begin{aligned}
\cL u+\nabla p=D_\alpha f_\alpha \quad \text{in }\, B_6,\\
\operatorname{div}u=g \quad \text{in }\, B_6,
\end{aligned}
\right.
\end{equation}
where $B_6=B_6(0)$ is the Euclidean ball in $\bR^d$ of radius $6$ centered at the origin.
The elliptic operator $\cL$ is in divergence form acting on column vector-valued functions $u=(u^1,\ldots,u^d)^{\top}$ as follows:
$$
\cL u=D_\alpha (A^{\alpha\beta}D_\beta u),
$$
where we use the Einstein summation convention on repeated indices.
Recently, $L^{q}$ theory of the Stokes system \eqref{171219@eq5a} was studied in \cite{arXiv:1604.02690v2, MR3758532, MR3809039}.
A consequence of those papers is that every weak solutions $(u, p)$ of the system \eqref{171219@eq5a} satisfy
$$
(u,p)\in W^{1,q}(B_1)^{d}\times  L^q(B_1)
$$
provided that the coefficients $A^{\alpha\beta}$ are merely measurable in one direction and have small mean oscillations in the other directions (the partially BMO condition), and that the data $f_\alpha$ and $g$ are in $L^q(B_6)$.
We note that the above $L^q$-regularity result holds when $q\in (1,\infty)$.
In this paper, we investigate minimal regularity assumption of the coefficients, which guarantees $L^\infty$-regularity for $Du$ and $p$.
We are also interested in $C^1$, weak type-$(1,1)$, and partial Schauder estimates for the Stokes system.

Throughout the paper, the coefficients $A^{\alpha\beta}=A^{\alpha\beta}(x)$ are $d\times d$ matrix-valued functions on $\bR^d$ satisfying the strong ellipticity condition; see \eqref{ell}.
We assume that $A^{\alpha\beta}$ are merely measurable in $x_1$-direction and the $L^1$-mean oscillation of $A^{\alpha\beta}$ with respect to $x'=(x_2,\ldots,x_d)$ satisfies the {\emph{Dini condition}}
\begin{equation}		\label{171219@eq5}
\int_0^1 \frac{\omega(r)}{r}\,dr<\infty.
\end{equation}
In this case, $A^{\alpha\beta}$ is called of {\emph{partially Dini mean oscillation}}; see Definition \ref{D1} for more precise definition.
As mentioned in \cite{arXiv:1604.02690v2}, such type of coefficients with no regularity assumption in one direction can be used to describe the motion of two or multiple fluids with interfacial boundaries.

This paper has two parts.
In Part \ref{Part1}, we are concerned with the interior regularity of weak solutions to the Stokes system.
We prove that if the coefficients and data are of partially Dini mean oscillation, then any weak solution to the Stokes system \eqref{171219@eq5a} satisfies
\begin{equation}		\label{171224@eq1}
(u, p)\in W^{1,\infty}(B_1)^{d}\times L^\infty(B_1).
\end{equation}
In particular, we show that certain linear combinations satisfy
\begin{equation}		\label{171226@eq1}
\hat{U}, D_{\alpha}u\in C(\overline{B_1})^d, \quad \alpha\in \{2,\ldots,d\},
\end{equation}
where $\hat{U}=A^{1\beta}D_\beta u+pe_1-f_1$ and $e_1$ is the first unit vector in $\bR^d$.
These linear combinations are H\"older continuous if we further assume that the coefficients and data are partially H\"older continuous; see Theorem \ref{M1}.
An immediate consequence of the results is that any weak solution to the Stokes system satisfies
$$
(u, p)\in C^1(\overline{B_1})^{d}\times C(\overline{B_1}) \quad \left(\text{resp. }\, (u, p)\in C^{1,\gamma_0}(\overline{B_1})^{d}\times C^{\gamma_0}(\overline{B_1})\right)
$$
provided that the coefficients and data are of Dini mean oscillation (resp. H\"older continuous) in {\emph{all}} directions; see Theorem \ref{M0}.
We note that our regularity results hold for $W^{1,1}$-weak solutions; see Theorem \ref{M3} and Remark \ref{171218@rmk1}.
With regard to previous results on $C^{1,\gamma_0}$-regularity for  $W^{1,2}$-weak solutions to linear or nonlinear Stokes systems, see  \cite{MR0641818}.
We also prove a weak type-$(1,1)$ estimate for $(Du, p)$ under the assumption that $\omega(r)\sim (\ln r)^{-2}$ for any small $r$, where $\omega$ is the $L^1$-mean oscillation of $A^{\alpha\beta}$ with respect to $x'$; see Theorem \ref{M2}.

In Part \ref{Part2}, we consider the corresponding boundary regularity results on the half ball $B_6^+=B_6\cap \bR^d_+$, where $\bR^d_+=\{x=(x_1,x')\in \bR^d:x_1>0, \, x'\in \bR^{d-1}\}$.
The main result is that if $(u, p)$ is a weak solution of
$$
\left\{
\begin{aligned}
\cL u+\nabla p=D_\alpha f_\alpha &\quad \text{in }\, B_6^+,\\
\operatorname{div} u=g &\quad \text{in }\, B_6^+,\\
u=0 &\quad \text{on }\, B_6\cap \partial \bR^d_+,
\end{aligned}
\right.
$$
where the coefficients and data are of partially Dini mean oscillation, then we have
\begin{equation}		\label{171226@A2}
\begin{aligned}
&(u, p)\in W^{1,\infty}(B_1^+)^{d}\times L^\infty(B_1^+),\\
&\hat{U}, D_{\alpha}u\in C(\overline{B_1^+})^d, \quad \alpha\in \{2,\ldots,d\}.
\end{aligned}
\end{equation}
As consequences of this result, we get $C^1$ and weak type-$(1,1)$ estimates for the solution.
For more precise statements of the boundary regularity results, see Section \ref{Sec5} in Part \ref{Part2}.

For the elliptic equation in divergence form
$$
D_i(a^{ij}D_j u)=0,
$$
it is well known that any weak solution $u$ is continuously differentiable provided that the coefficients $a^{ij}=a^{ij}(x)$ satisfy the $\alpha$-increasing Dini continuity condition for some $\alpha\in (0,1]$; see, for instance, \cite{MR0521856,MR0932759}.
In \cite{MR3615500}, Li proved that weak solutions of the elliptic equation (or system) are continuously differentiable when the modulus of continuity of coefficients in the $L^\infty$ sense satisfies the Dini condition \eqref{171219@eq5}.
This regularity result was extended by Dong-Kim \cite{MR3620893} to the case of
$$
D_i(a^{ij}D_j u)=\operatorname{div}f,
$$
where the coefficients and data are of Dini mean oscillation.
They also established weak type-$(1,1)$ estimates for derivatives of solutions under a stronger assumption on the mean oscillation, and
the corresponding results for nondivergence form equations.
Later, Dong-Escauriaza-Kim \cite{MR3747493} extended the results in \cite{MR3620893} up to the boundary for the solutions satisfying the Dirichlet boundary condition.
We also refer the reader to \cite{MR2927619, MR3403998}, where the authors considered parabolic and elliptic systems with partially Dini or H\"older continuous coefficients. We note that partial Schauder estimates for elliptic equations were studied long time ago by Fife \cite{MR0158162}. See also \cite{MR2644772,MR2764915,MR3374066,DK15b} and the references therein for some recent work in this direction.

Our arguments in establishing \eqref{171224@eq1} and \eqref{171226@eq1} are based on techniques, called Campanato's approach, used in \cite{MR2927619, MR3620893}. The main step of Campanato's approach is to show that, for example,  in the case of elliptic equations in divergence form, the mean oscillations of $Du$ in balls vanish in certain order as the radii of balls go to zero.
For this, in \cite{MR3620893}, the authors utilized weak type-$(1,1)$ estimates for elliptic equations with constant coefficients.
However, we are not able to follow this approach in the same way to control the mean oscillation of $Du$ because we only impose the assumption on the $L^1$-mean oscillation of the coefficients and data with respect to $x'$.
This assumption causes the lack of regularity of $Du$ and $p$ in the $x_1$-direction.
To overcome this difficulty, we exploit an idea in \cite{MR2927619} to refine the argument in \cite{MR3620893} in the setting of the Stokes system with coefficients measurable in $x_1$-direction, and then apply the weak type estimate to control the $L^{q}$-mean oscillations of $\hat{U}=A^{1\beta}D_\beta u+p e_1-f_1$ and $D_{x'}u$ for $q\in (0,1)$.
This allows us to get the desired results \eqref{171224@eq1} and \eqref{171226@eq1}.
We point out that based on our argument, one can investigate  divergence form elliptic equations and systems with coefficients having partially Dini mean oscillation.

For the boundary regularity \eqref{171226@A2}, we adapt the aforementioned arguments and the techniques used in \cite{MR3747493}, where the authors proved boundary $C^1$-estimates for elliptic equations with coefficients having Dini mean oscillation.
We note that in \cite{MR3747493}, they established the boundary estimates in a half ball as well as in a domain $\Omega$ with $C^{1,\rm{Dini}}$ boundary, i.e., $\partial\Omega$ is locally the graph of a $C^1$ function whose derivatives are uniformly Dini continuous.
In fact, in this case, the estimate in $\Omega$ is an easy consequence of that in the half ball because the mapping of flattening boundary preserves the regularity assumptions on the coefficients and data.

In a subsequent paper, we will study the boundary regularity of weak solutions to the Stokes system with coefficients having Dini mean oscillation (in all directions) in a domain with $C^{1,\rm{Dini}}$ boundary. It turns out that in this case, the proof is more involved than in the case of elliptic equations because of the pressure term and the divergence equation in the Stokes system.

The $L^\infty$ and $C^{\alpha}$ estimates for Stokes systems play an important role in the study of Green functions.
In \cite{MR3693868}, the authors constructed the Green function of the Stokes system for a representation formula of the flow velocity $u$ by using the $C^\alpha$-estimate for the system.
Based on our results regarding the $L^\infty$-estimate for the pressure $p$, we will discuss the Green function of the Stokes system for a representation formula of the pressure $p$ in a forthcoming paper.

The remainder of this paper is divided into two parts and organized as follows.
The first part is devoted to the interior regularity.
In Section \ref{Sec2}, we state our main results.
In Section \ref{Sec3}, we provide some preliminary lemmas, and in Section \ref{Sec4}, we prove the main results on the interior regularity.
The second part is devoted to the boundary regularity.
In Section \ref{Sec5}, we state the main results.
We prove some preliminary lemmas in Section \ref{Sec6}.
The proofs of main theorems regarding the boundary regularity are given in Section \ref{Sec7}.
In Appendix, we provide the proofs of some technical lemmas.

\part{Interior estimates}		\label{Part1}
This part of the paper is devoted to the interior estimates for Stokes systems.

\section{Main results}		\label{Sec2}

We first fix some notation used throughout the paper.
We use $x=(x_1,x')$ to denote a point in $\bR^d$, where $d\ge 2$.
We also write $y=(y_1,y')$ and $x_0=(x_{01},x_0')$, etc.
Set
$$
B_r(x)=\{y\in \bR^d:|x-y|<r\}, \quad B_r'(x')=\{y'\in \bR^{d-1} : |x'-y'|<r\}.
$$
We use the abbreviations $B_r=B_r(0)$ where $0\in \bR^d$,  and $B_r'=B_r'(0')$ where $0'\in \bR^{d-1}$.
For $k\in \{1,\ldots,d\}$, we denote by $e_k$ the $k$-th unit vector in $\bR^d$.

For $0 < q \le \infty$, let $L^q(\Omega)$ be the space consisting of measurable functions on $\Omega$ that are $q$-th
integrable.
We recall that
$$
\|f+g\|_{L^q(\Omega)}\le C_q \big(\|f\|_{L^q(\Omega)}+\|g\|_{L^q(\Omega)}\big),
$$
where $C_q=2^{(1-q)/q}$ if $0<q<1$ and $C_q=1$ if $1\le q\le \infty$.
When $|\Omega|<\infty$, we define
$$
\tilde{L}^q(\Omega)=\{f\in L^q(\Omega): (f)_\Omega=0\},
$$
where $(f)_\Omega$ is the average of $f$ over $\Omega$, i.e.,
$$
(f)_\Omega=\dashint_\Omega f\,dx=\frac{1}{|\Omega|}\int_\Omega f\,dx.
$$
For $1\le q\le \infty$, we denote by $W^{1,q}(\Omega)$ the usual Sobolev space and by $W^{1,q}_0(\Omega)$ the completion of $C^\infty_0(\Omega)$ in $W^{1,q}(\Omega)$.

For $0<\gamma<1$, the H\"older semi-norm is defined by
$$
[f]_{C^{\gamma}(\Omega)}:=\sup_{\substack{x,y\in \Omega \\ x\neq y}} \frac{|f(x)-f(y)|}{|x-y|^\gamma},
$$
and the partial H\"older semi-norm with respect to $x'$ is defined by
$$
[f]_{C^\gamma_{x'}(\Omega)}:=\sup_{\substack{x,y\in \Omega \\ x_1=y_1,\, x'\neq y'}}\frac{|f(x)-f(y)|}{|x'-y'|^\gamma}.
$$

Let $\cL$ be a strongly elliptic operator of the form
$$
\cL u=D_\alpha (A^{\alpha\beta}D_\beta u),
$$
where the coefficients $A^{\alpha\beta}=A^{\alpha\beta}(x)$ are $d\times d$ matrix-valued functions on $\bR^d$ with
entries $A_{ij}^{\alpha\beta}$ satisfying the strong ellipticity condition,
i.e., there is a constant $\lambda\in (0,1]$ such that for any $x\in \bR^d$ and $ \xi_\alpha \in \bR^d$, $\alpha\in \{1,\ldots,d\}$, we have
\begin{equation}		\label{ell}
|A^{\alpha\beta}(x)|\le \lambda^{-1}, \quad \sum_{\alpha,\beta=1}^dA^{\alpha\beta}(x)\xi_\beta\cdot \xi_\alpha\ge \lambda \sum_{\alpha=1}^d|\xi_\alpha|^2.
\end{equation}
For $f_\alpha\in L^q(\Omega)^d$ with $q\ge 1$, we say that $(u,p)\in W^{1,q}(\Omega)^d\times L^q(\Omega)$ is a weak solution of
$$
\cL u+\nabla p=D_\alpha f_\alpha \quad \text{in }\, \Omega,
$$
if
$$
\int_\Omega A^{\alpha\beta}D_\beta u\cdot D_\alpha \phi \, dx+\int_\Omega p \operatorname{div}\phi\,dx =\int_\Omega f_\alpha \cdot D_\alpha \phi \,dx
$$
holds for $\phi\in C^\infty_0(\Omega)^d$.

\begin{definition}		\label{D1}
Let $f\in L^1(B_6)$.
\begin{enumerate}[$(i)$]
\item
We say that $f$ is of  {\em{Dini mean oscillation in $B_4$}} if the function $\omega_f:(0,1]\to [0,\infty)$ defined by
$$
\omega_{f}(r):=\sup_{x\in B_4} \dashint_{B_r(x)}\bigg|f(y)-\dashint_{B_r(x)}f(z)\, dz\bigg| \,dy
$$
satisfies
$$
\int_0^1 \frac{\omega_f(r)}{r}\,dr <\infty.
$$
\item
We say that $f$ is of {\em{partially Dini mean oscillation with respect to $x'$ in $B_4$}} if the function $\omega_{f,x'}:(0,1]\to [0,\infty)$ defined by
$$
\omega_{f,x'}(r):=\sup_{x\in B_4} \dashint_{B_r(x)}\bigg|f(y)-\dashint_{B_r'(x')}f(y_1,z')\, dz'\bigg| \,dy
$$
satisfies
$$
\int_0^1 \frac{\omega_{f,x'}(r)}{r}\,dr<\infty.
$$
\end{enumerate}

\end{definition}

The main results of the first part of the paper read as follows.

\begin{theorem}		\label{M1}
Let $q_0\in (1,\infty)$.
Assume that $(u,p)\in W^{1,q_0}(B_6)^d\times L^{q_0}(B_6)$ is a weak solution of
\begin{equation}		\label{170907_eq1}
\left\{
\begin{aligned}
\cL u+\nabla p=D_\alpha f_\alpha &\quad \text{in }\, B_6,\\
\operatorname{div} u=g &\quad \text{in }\, B_6,
\end{aligned}
\right.
\end{equation}
where $f_1\in L^\infty(B_6)^d$, $f_\alpha\in L^{q_0}(B_6)^d$, $ \alpha\in \{2,\ldots,d\}$, and $g\in L^\infty(B_6)$.
Set
$$
\hat{U}:=A^{1\beta}D_\beta u+pe_1-f_1.
$$
\begin{enumerate}[$(a)$]
\item
If $A^{\alpha\beta}$,  $f_\alpha$, and $g$ are of partially Dini mean oscillation with respect to $x'$ in $B_4$,
then we have
\begin{equation}		\label{170927@eq5}
(u,p)\in W^{1,\infty}(B_1)^d\times L^\infty(B_1)
\end{equation}
and
$$
\hat{U}, D_\alpha u\in C(\overline{B_1})^d, \quad \alpha\in \{2,\ldots,d\}.
$$
\item
If it holds that $[A^{\alpha\beta}]_{C^{\gamma_0}_{x'}(B_6)}+[f_\alpha]_{C^{\gamma_0}_{x'}(B_6)}+[g]_{C^{\gamma_0}_{x'}(B_6)}<\infty$ for some $\gamma_0\in (0,1)$, then we have \eqref{170927@eq5} and
$$
\hat{U}, D_\alpha u\in C^{\gamma_0}(\overline{B_1})^d, \quad \alpha\in \{2,\ldots,d\}.
$$
\end{enumerate}
\end{theorem}

\begin{theorem}		\label{M0}
Let $q_0\in (1,\infty)$.
Assume that  $(u,p)\in W^{1,q_0}(B_6)^d\times {L}^{q_0}(B_6)$ is a weak solution of \eqref{170907_eq1}, where $f_\alpha\in L^{q_0}(B_6)^d$ and $g\in L^{q_0}(B_6)$.
\begin{enumerate}[$(a)$]
\item
If $A^{\alpha\beta}$,  $f_\alpha$, and $g$ are of Dini mean oscillation in $B_4$,
then we have
$$
(u,p)\in C^1(\overline{B_1})^d\times C(\overline{B_1}).
$$
\item
If it holds that $[A^{\alpha\beta}]_{C^{\gamma_0}(B_6)}+[f_\alpha]_{C^{\gamma_0}(B_6)}+[g]_{C^{\gamma_0}(B_6)}<\infty$ for some $\gamma_0\in (0,1)$, then we have
$$
(u,p)\in C^{1,{\gamma_0} }(\overline{B_1})^d\times C^{\gamma_0}(\overline{B_1}).
$$
\end{enumerate}
\end{theorem}

Theorem \ref{M0} is a special case of Theorem \ref{M1} because any function that is of Dini mean oscillation (resp. H\"older continuous) is of partially Dini mean oscillation (resp. partially H\"older continuous) with respect to any $(d-1)$-directions.
However, we provide in Section \ref{170929@sec1} the major steps of the proof of Theorem \ref{M0} showing that the estimates do not involve the $L^\infty$-norms of data.
On the other hand, the estimates in the proof of Theorem \ref{M1} include the $L^\infty$-norms of $f_1$ and $g$.

\begin{remark}		\label{171008@rmk1}
One can easily extend the results in Theorems \ref{M1} and \ref{M0} to a solution $(u,p)$ of the system
$$
\left\{
\begin{aligned}
\cL u+\nabla p=f+D_\alpha f_\alpha &\quad \text{in }\, B_6,\\
\operatorname{div} u=g &\quad \text{in }\, B_6,
\end{aligned}
\right.
$$
where $f\in L^q(B_6)^d$ with $q>d$.
Indeed, by \cite[Lemma 3.1]{MR3809039}, there exist $\hat{f}_\alpha\in W^{1,q}(B_6)^d$, $\alpha\in \{1,\ldots,d\}$, such that $D_\alpha \hat{f}_\alpha=f$, which implies that $(u,p)$ is a weak solution of \eqref{170907_eq1} with $\hat{f}_\alpha+f_\alpha$ in place of $f_\alpha$.
Moreover, by the Morrey-Sobolev inequality, we have that $\hat{f}_\alpha\in C^{\gamma_0}(\overline{B_6})^d$ with $\gamma_0=1-\frac{d}{q}>0$, and thus $\hat{f}_\alpha$ are of Dini mean oscillation.
\end{remark}

Based on a duality argument and Theorem \ref{M1} $(a)$,
we obtain the following $L^{q_0}$-estimate for $W^{1,1}$-weak solutions.

\begin{theorem}		\label{M3}
Let $q_0\in (1,\infty)$.
Assume that $(u,p)\in W^{1,1}(B_6)^d\times L^1(B_6)$ is a weak solution of \eqref{170907_eq1}, where $f_\alpha\in L^{q_0}(B_6)^d$ and $g\in L^{q_0}(B_6)$.
If $A^{\alpha\beta}$ are of partially Dini mean oscillation with respect to $x'$ in $B_4$, then we have $(u,p)\in W^{1,q_0}(B_1)^d\times L^{q_0}(B_1)$ with the estimate
\begin{align*}
&\|u\|_{W^{1,q_0}(B_1)}+\|p\|_{L^{q_0}(B_1)} \\
&\le C\big( \|u\|_{W^{1,1}(B_6)}+\|p\|_{L^1(B_6)}+\|f_\alpha\|_{L^{q_0}(B_6)}+\|g\|_{L^{q_0}(B_6)}\big),
\end{align*}
where the constant $C$ depends only on $d$, $\lambda$, $\omega_{A^{\alpha\beta},x'}$, and $q_0$.
\end{theorem}

\begin{remark}		\label{171218@rmk1}
By Theorem \ref{M3} and a covering argument, we can see that the results in Theorems \ref{M1} and \ref{M0} still hold under the assumption that $(u,p)\in W^{1,1}(B_6)^d\times L^1(B_6)$.
\end{remark}

We also prove the following local weak type-$(1,1)$ estimate.

\begin{theorem}		\label{M2}
Define a bounded linear operator $T_0$ on $L^2(B_6)^{d\times d}\times L^2(B_6)$ by
$$
T_0(f_1,\ldots,f_d,g)=(D_1u, \ldots, D_du,p),
$$
where $(u, p)\in W^{1,2}_0(B_6)^d\times \tilde{L}^2(B_6)$ is a unique weak solution of
\begin{equation}		\label{171122@eq2}
\left\{
\begin{aligned}
\cL u +\nabla p=D_\alpha f_\alpha &\quad \text{in }\, B_6,\\
\operatorname{div} u=g-(g)_{B_6} &\quad \text{in }\, B_6.
\end{aligned}
\right.
\end{equation}
If $A^{\alpha\beta}$ are of partially Dini mean oscillation with respect to $x'$ in $B_4$ and
$$
\omega_{A^{\alpha\beta},x'}(r)\le C_0 \bigg(\ln \frac{r}{4}\bigg)^{-2}, \quad \forall r\in (0,1],
$$
then the operator $T_0$ can be extended on
$$
\big\{(f_1,\ldots,f_d,g)\in L^1(B_6)^{d\times d}\times L^1(B_6): \operatorname{supp}(f_1,\ldots,f_d,g)\subset B_1\big\}
$$
such that for any $t>0$, we have
$$
\big|\{x\in B_1: |T_0 (f_1,\ldots,f_d,g)|>t\}\big|\le \frac{C}{t}\int_{B_1}\big(|f_\alpha|+|g|\big)\,dx,
$$
where the constant $C$ depends only on $d$, $\lambda$, $\omega_{A^{\alpha\beta},x'}$, and $C_0$.
\end{theorem}

\section{Preliminary lemmas}		\label{Sec3}
In this section, we prove some preliminary results which will be used in the proofs of the main theorems in the next section.
Throughout this section, we set
$$
\cL_0 u=D_\alpha(\bar{A}^{\alpha\beta}D_\beta u),
$$
where $\bar{A}^{\alpha\beta}=\bar{A}^{\alpha\beta}(x_1)$ satisfy \eqref{ell}.
Hereafter in the paper,  we shall use the following notation.
\begin{notation}
For nonnegative (variable) quantities $A$ and $B$,
we denote $A\lesssim B$ if there exists a generic positive constant C such that $A \le CB$.
We add subscript letters like $A\lesssim_{a,b} B$ to indicate the dependence of the implicit constant $C$ on the parameters $a$ and $b$.
\end{notation}

The following lemma is regarding Lipschitz estimates of $u$ and linear combinations of $Du$ and $p$.

\begin{lemma}		\label{170901@lem1}
Let $0<r<R$ and $\ell$ be a constant.
Assume that $(u,p)\in W^{1,2}(B_R)^d\times L^2(B_R)$ satisfies
\begin{equation}		\label{170908@eq2}
\left\{
\begin{aligned}
\cL_0 u+\nabla p=0 & \quad \text{in }\, B_R,\\
\operatorname{div} u= \ell & \quad \text{in }\, B_R.
\end{aligned}
\right.
\end{equation}
Then we have
\begin{align}
\label{170908@eq1}
\|Du\|_{L^\infty(B_r)} &\lesssim_{d,\lambda} (R-r)^{-d/2}\|Du\|_{L^2(B_R)},\\
\label{170908@eq1a}
[U]_{C^{0,1}(B_{r})}+[D_{x'}u]_{C^{0,1}(B_{r})} &\lesssim_{d,\lambda} (R-r)^{-d/2-1}\|Du\|_{L^2(B_R)},
\end{align}
where $U:=\bar{A}^{1\beta}D_\beta u+p e_1$.
Here, we denote the $C^{0,1}$ semi-norm by
\begin{equation}		\label{171218@eq1}
[f]_{C^{0,1}(B_r)}=\sup_{\substack{x,y\in B_r \\ x\neq y}}\frac{|f(x)-f(y)|}{|x-y|}.
\end{equation}
\end{lemma}

\begin{proof}
By \cite[Lemma 4.1 $(i)$]{MR3758532}, we have \eqref{170908@eq1} and
\begin{equation}		\label{171210@A1}
\|p\|_{L^\infty(B_r)}\lesssim_{d,\lambda} (R-r)^{-d/2}\big(\|Du\|_{L^2(B_R)}+\|p\|_{L^2(B_R)}\big).
\end{equation}
We also have  from \cite[Lemmas 4.1 and 4.2]{arXiv:1604.02690v2} that
$$
(D_{i}u, D_i p)\in W^{1,2}(B_{(R+r)/2})^d\times L^2(B_{(R+r)/2}), \quad i\in \{2,\ldots,d\},
$$
and
$$
\|D(D_{x'}u)\|_{L^2(B_{(R+r)/2})}+\|D_{x'}p\|_{L^2(B_{(R+r)/2})}\lesssim_{d,\lambda} (R-r)^{-1}\|Du\|_{L^2(B_{R})}.
$$
Since $(D_{x'}u, D_{x'} p)$ satisfies \eqref{170908@eq2} with $\ell=0$ in $B_{(R+r)/2}$, by applying \eqref{170908@eq1}, \eqref{171210@A1}, and the above inequality, we obtain
\begin{align}
\nonumber
&\|D(D_{x'}u)\|_{L^\infty(B_r)}+\|D_{x'}p\|_{L^\infty(B_r)} \\
\nonumber
&\lesssim (R-r)^{-d/2}\big(\|D(D_{x'}u)\|_{L^2(B_{(R+r)/2})}+\|D_{x'}p\|_{L^2(B_{(R+r)/2})}\big)\\
\label{171210@A3}
&\lesssim (R-r)^{-d/2-1}\|Du\|_{L^2(B_{R})}.
\end{align}
This yields that
$$
\|D_{x'}U\|_{L^\infty(B_r)}\lesssim (R-r)^{-d/2-1}\|Du\|_{L^2(B_R)}.
$$
Moreover, since we have that (using $\cL_0 u+\nabla p=0$)
$$
D_1 U=-\sum_{\alpha=2}^d  (\bar{A}^{\alpha\beta}D_{\alpha \beta} u+D_\alpha p e_\alpha),
$$
the estimate \eqref{171210@A3} implies
$$
\|D_1 U\|_{L^\infty (B_r)}\lesssim (R-r)^{-d/2-1}\|Du\|_{L^2(B_R)}.
$$
Combining the above inequalities, we have
$$
\|DU\|_{L^\infty(B_r)}+\|D(D_{x'}u)\|_{L^\infty(B_r)}\lesssim (R-r)^{-d/2-1}\|Du\|_{L^2(B_R)},
$$
which gives \eqref{170908@eq1a}.
The lemma is proved.
\end{proof}

In the next lemma, we obtain $L^q$-mean oscillation estimates of linear combinations of $Du$ and $p$ for $q\in (0,1)$.

\begin{lemma}		\label{170908@lem1}
Let $0<r\le R/2$ and $\ell$ be a constant.
Assume that $(u,p)\in W^{1,2}(B_R)^d\times L^2(B_R)$ satisfies \eqref{170908@eq2}.
Then for any $q\in (0,1)$, we have
\begin{equation}		\label{171110@eq1}
\begin{aligned}
&\left(\dashint_{B_r}|U-(U)_{B_r}|^q +|D_{x'}u -(D_{x'}u)_{B_r}|^q \, dx\right)^{1/q}\\
&\lesssim_{d,\lambda,q}
\frac{r}{R}
\inf_{\substack{\theta\in \bR^d \\ \Theta\in \bR^{d\times (d-1)}}}
\left(\dashint_{B_R}|U-\theta |^q +|D_{x'}u -\Theta |^q \, dx\right)^{1/q},
\end{aligned}
\end{equation}
where $U:=\bar{A}^{1\beta}D_\beta u+p e_1$.
If $\cL_0$ has constant coefficients, then we have
\begin{equation}		\label{171110@eq1a}
\begin{aligned}
&\left(\dashint_{B_r}|Du-(Du)_{B_r}|^q +|p -(p)_{B_r}|^q \, dx\right)^{1/q}\\
&\lesssim_{d,\lambda,q}
\frac{r}{R}
\inf_{\Theta\in \bR^{d\times d}}
\left(\dashint_{B_R}|Du -\Theta |^q \, dx\right)^{1/q}.
\end{aligned}
\end{equation}
\end{lemma}

\begin{proof}
To show \eqref{171110@eq1}, we first claim that if $(u,p)\in W^{1,2}(B_R)^d\times L^2(B_R)$ satisfies \eqref{170908@eq2} with $\ell=0$, then we have for $0<r\le R/2$ that
\begin{equation}		\label{170913@eq1}
\begin{aligned}
&\left(\dashint_{B_r}|U-(U)_{B_r}|^q +|D_{x'}u -(D_{x'}u)_{B_r}|^q \, dx\right)^{1/q}\\
&\lesssim_{d,\lambda,q}
\frac{r}{R}
 \left(\dashint_{B_R} |U|^q+|D_{x'}u|^q\, dx \right)^{1/q}.
\end{aligned}
\end{equation}
Suppose that $(u,p)\in W^{1,2}(B_R)^d\times L^2(B_R)$ satisfies \eqref{170908@eq2} with $\ell=0$.
By \eqref{170908@eq1} and a standard iteration argument (see, for instance, \cite[pp. 80--82]{MR1239172}), we obtain for $0<\rho< R$ that
$$
\|Du\|_{L^\infty(B_\rho)} \lesssim_{d,\lambda,q} (R-\rho)^{-d/q}\|Du\|_{L^q(B_R)}.
$$
Set $\rho=\frac{r+R}{2}$.
From \eqref{170908@eq1a} and the above inequality, it follows that
\begin{align*}
[U]_{C^{0,1}(B_{r})}+[D_{x'}u]_{C^{0,1}(B_{r})}&\lesssim (\rho-r)^{-d/2-1}\|Du\|_{L^2(B_\rho)}\\
&\lesssim (\rho-r)^{-d/2-1}\|Du\|_{L^\infty(B_\rho)}^{(2-q)/2}\|Du\|_{L^q(B_\rho)}^{q/2}\\
&\lesssim (R-r)^{-d/q-1}\|Du\|_{L^q(B_R)}.
\end{align*}
Hence we have
\begin{equation}		\label{170908@eq5}
\begin{aligned}
&\left(\dashint_{B_r}|U-(U)_{B_r}|^q +|D_{x'}u -(D_{x'}u)_{B_r}|^q \, dx\right)^{1/q}\\
& \lesssim \frac{r}{R} \left(\dashint_{B_R} |Du|^q \, dx \right)^{1/q}.
\end{aligned}
\end{equation}
Notice from the definition of $U$ that
$$
\sum_{j=2}^d \bar{A}^{11}_{ij} D_1 u^j=U^i - \sum_{j=1}^d \sum_{\beta=2}^d \bar{A}^{1\beta}_{ij}D_\beta u^j-\bar{A}^{11}_{i1}D_1 u^1, \quad i\in \{2,\ldots,d\}.
$$
By the ellipticity condition on $\bar{A}^{\alpha\beta}$, $[\bar{A}^{11}_{ij}]_{i, j=2}^d$ is nondegenerate.
Thus using the fact that  $D_1u^1= -\sum_{j=2}^d D_j u^j$,
we have
$$
|Du|\le |D_1 u^1 | + \sum_{j=2}^d |D_1 u^j |+|D_{x'}u| \lesssim |U|+|D_{x'}u|.
$$
This together with \eqref{170908@eq5} gives \eqref{170913@eq1}.

Now we are ready to prove that \eqref{171110@eq1} holds.
Suppose that  $(u,p)\in W^{1,2}(B_R)^d\times L^2(B_R)$ satisfies \eqref{170908@eq2}.
For given $\theta_i\in \bR^d$, $i\in \{1,\ldots,d\}$, find functions $u_0=(u_0^1(x_1),\ldots,u^d_0(x_1))^{\top}$ and $p_0=p_0(x_1)$ such that
$$
u_0^1=-\sum_{i=2}^d \theta_i +\ell, \quad \bar{A}^{11}u_0+p_0 e_1=\theta_1-\sum_{\beta=2}^d \bar{A}^{1\beta}\theta_\beta.
$$
A direct calculation shows that the pair $(u_e, p_e)$ given by
$$
u_e=u-\int^{x_1}_{-R} u_0 \, dy_1-\sum_{i=2}^d x_i \theta_i \quad \text{and} \quad p_e=p-p_0
$$
satisfies \eqref{170908@eq2} with $\ell=0$.
Therefore, by applying \eqref{170913@eq1} and using the fact that
$$
\bar{A}^{1\beta}D_\beta u_e+p_e e_1=U-\theta_1, \quad D_{i}u_e=D_{i}u-\theta_i, \quad i\in \{2,\ldots,d\},
$$
we have
$$
\begin{aligned}
&\left(\dashint_{B_r}|U-(U)_{B_r}|^q +|D_{x'}u -(D_{x'}u)_{B_r}|^q \, dx\right)^{1/q}\\
&\lesssim \frac{r}{R}
\left(\dashint_{B_R}|U-\theta_1 |^q +\sum_{i=2}^d|D_{i}u -\theta_i |^q \, dx\right)^{1/q}.
\end{aligned}
$$
Since the above inequality holds for arbitrary $\theta_i \in \bR^d$, we get the estimate \eqref{171110@eq1}.

Next we turn to the proof of \eqref{171110@eq1a}.
Assume that $(u,p)\in W^{1,2}(B_R)^d\times L^2(B_R)$ satisfies \eqref{170908@eq2}, where $\cL_0$ has constant coefficients.
Then by using \eqref{170908@eq1a} and relabeling the coordinates axes, we have for $0<r<R$ that
$$
[Du]_{C^{0,1}(B_r)}+[p]_{C^{0,1}(B_r)}\lesssim (R-r)^{-d/2-1}\|Du\|_{L^2(B_R)}.
$$
Thus by following the same argument used in deriving \eqref{170908@eq5}, we see that
\begin{equation}		\label{171110@eq2}
\begin{aligned}
&\left(\dashint_{B_r}|Du-(Du)_{B_r}|^q+|p-(p)_{B_r}|^q\,dx\right)^{1/q}\lesssim
\frac{r}{R}
 \left(\dashint_{B_R}|Du|^q\,dx\right)^{1/q}
\end{aligned}
\end{equation}
for all $0<r\le R/2$.
Let $\theta_1,\ldots, \theta_d\in \bR^d$, and set $u_e=u-\sum_{i=1}^d \theta_i x_i$.
Then $(u_e,p)$ satisfies
$$
\left\{
\begin{aligned}
\cL_0 u_e+\nabla p=0 & \quad \text{in }\, B_R,\\
\operatorname{div} u_e= \ell-\sum_{i=1}^d \theta_i^i & \quad \text{in }\, B_R.
\end{aligned}
\right.
$$
Therefore, by \eqref{171110@eq2} with $(u_e,p)$ in place of $(u,p)$, we have
$$
\left(\dashint_{B_r}|Du-(Du)_{B_r}|^q+|p-(p)_{B_r}|^q\,dx\right)^{1/q}\lesssim
\frac{r}{R}
\left(\dashint_{B_R}\sum_{i=1}^d|D_iu-\theta_i|^q\,dx\right)^{1/q},
$$
where we used the fact that $D_i u_e=D_i u-\theta_i$.
Since the above inequality holds for arbitrary $\theta_i\in \bR^d$, we get \eqref{171110@eq1a}.
The lemma is proved.
\end{proof}

\begin{lemma}		\label{170823@lem1}
Let $T$ be a bounded linear operator from $L^2(B_{R_0})^{k}$ to $L^2(B_{R_0})^k$, where $R_0\ge 4$ and $k\in \{1,2,\ldots\}$.
Let $\mu<1$, $c>1$, and $C>0$ be constants.
Suppose that for any $x_0\in B_1$, $0 < r < \mu$, and
$$
g\in \tilde{L}^2(B_{R_0})^k \,\text{ with }\, \operatorname{supp} \, g \subset B_r(x_0)\cap B_1,
$$
we have
\begin{equation}		\label{180813@eq1}
\int_{B_1\setminus B_{cr}(x_0)} |T g| \, dx \le C \int_{ B_r(x_0)\cap B_1} |g| \, dx.
\end{equation}
Then the operator $T$ can be extended on
$$
\big\{f\in L^1(B_{R_0})^k: \operatorname{supp}f\subset B_1\big\}
$$
 such that for any $t>0$,
we have
$$
\big|\{x\in B_1: |T f(x)|>t\}\big| \lesssim_{d,k,\mu,c,C} \frac{1}{t}\int_{B_1} | f | \, dx.
$$
\end{lemma}

\begin{proof}
Let $x_0\in B_1$ and $0<r<\mu$, and set
$$
M=\big\{g\in \tilde{L}^1(B_{R_0})^k :\operatorname{supp}g\subset B_r(x_0)\cap B_1\big\}.
$$
Then by the hypothesis of the lemma, the operator $T$ can be extended as a bounded linear operator from  $M$ to $L^1(B_1\setminus B_{cr}(x_0))^k$.
Indeed, for a given $g\in M$, one can find a sequence $\{g_n\}\subset \tilde{L}^2(B_{R_0})^k$ such that
$$
\operatorname{supp}g_n\subset B_r(x_0)\cap B_1, \quad  \|g_n-g\|_{L^1(B_{R_0})}\to 0 \, \text{ as }\, n\to \infty.
$$
Thus by \eqref{180813@eq1}, $\{Tg_n\}$ is a Cauchy sequence in $L^1(B_1\setminus B_{cr}(x_0))^k$ and its limit, denoted by $Tg$, satisfies \eqref{180813@eq1}.

By using the above extension of $T$ and following the proof of \cite[Lemma 4.1]{MR3747493} with $\Omega=B_1$, one can easily prove the lemma.
We omit the details here.
\end{proof}

We finish this section by establishing a weak type-$(1,1)$ estimate.

\begin{lemma}		\label{170904@lem1}
Let $(u,p)\in W^{1,2}_0(B_4)^d\times \tilde{L}^2(B_4)$ be the weak solution of
\begin{equation}		\label{170904@eq1}
\left\{
\begin{aligned}
\cL_0 u+\nabla p=D_\alpha {f}_\alpha & \quad \text{in }\, B_4,\\
\operatorname{div} u= {g}-(g)_{B_4} & \quad \text{in }\, B_4,
\end{aligned}
\right.
\end{equation}
where $f_\alpha\in L^2(B_4)^d$ and $g\in {L}^2(B_4)$ are supported in $B_1$.
Then for any $t>0$, we have
$$
\big|\{x\in B_{1}:|Du(x)|+|p(x)|>t\}\big| \lesssim_{d,\lambda} \frac{1}{t}\int_{B_1} \big(|f_\alpha|+|g|\big)\,dx.
$$
\end{lemma}

\begin{proof}
Denote $\cA=(\bar{A}^{11})^{-1}$.
Consider a linear map $F$ on $(L^1\cap L^2)(B_4)^{d\times d}\times (L^1\cap L^2)(B_4)$ given by $F(f_1,\ldots,f_d,g)=(\tilde{f}_1,\ldots,\tilde{f}_d, \tilde{g})$,
where
$$
\tilde{g}=\frac{1}{\cA_{11}} (\cA f_1 \cdot e_1- g), \quad \tilde{f}_1=\cA(f_1-\tilde{g}e_1), \quad \tilde{f}_\alpha=f_\alpha-\bar{A}^{\alpha 1} \tilde{f}_1, \quad \alpha\in \{2,\ldots,d\}.
$$
Then the map $F$ is bounded and invertible on $L^1(B_4)^{d\times d}\times L^1(B_4)$ and also on $L^2(B_4)^{d\times d}\times L^2(B_4)$.
Indeed, for given $(\tilde{f}_1,\ldots,\tilde{f}_d,\tilde{g})\in L^1(B_4)^{d\times d}\times L^1(B_4)$ (or $L^2(B_4)^{d\times d}\times L^2(B_4)$),
there exist uniquely determined functions
$$
g=\tilde{f}^1_1, \quad f_1=\bar{A}^{11}\tilde{f}_1+\tilde{g} e_1, \quad  f_\alpha=\bar{A}^{\alpha 1} \tilde{f}_1+\tilde{f}_\alpha, \quad \alpha\in \{2,\ldots,d\},
$$
such that $F(f_1,\ldots,f_d,g)=(\tilde{f}_1,\ldots,\tilde{f}_d,\tilde{g})$.
This implies that $F$ is bijective, and thus, $F$ is invertible.

Define a bounded linear operator $T$ on $L^2(B_4)^{d\times d}\times {L}^2(B_4)$ by
$$
T(\tilde{f}_1, \ldots, \tilde{f}_d, \tilde{g})=(D_1 u, \ldots, D_d u, p),
$$
where $(u,p)\in W^{1,2}_0(B_4)^d\times \tilde{L}^2(B_4)$ is the weak solution of \eqref{170904@eq1} with
$$
(f_1,\ldots,f_d,g)=F^{-1}(\tilde{f}_1,\ldots,\tilde{f}_d,\tilde{g}).
$$
To prove the lemma, it suffices to show that $T$ satisfies the hypothesis of Lemma \ref{170823@lem1} with $\mu=1/2$ and $c=2$.

Fix $x_0\in B_1$ and $0<r<1/2$, and let $\tilde{f}_\alpha\in \tilde{L}^2(B_4)^d$ and $\tilde{g}\in \tilde{L}^2(B_4)$ be supported in $B_r(x_0)$.
Let $(u,p)\in W^{1,2}_0(B_4)^d\times \tilde{L}^2(B_4)$ be the weak solution of \eqref{170904@eq1} with $(f_1,\ldots,f_d,g)=F^{-1}(\tilde{f}_1,\ldots,\tilde{f}_d,\tilde{g})$.
Then we have that
$$
\operatorname{div} u=\tilde{f}^1_1 \quad \text{in }\, B_4
$$
and
\begin{equation}		\label{170904@eq2}
\begin{aligned}
&\int_{B_4} \bar{A}^{\alpha\beta}D_\beta u\cdot D_\alpha \varphi \, dx+\int_{B_4}p \operatorname{div} \varphi\,dx\\
&=\sum_{\alpha=2}^d\int_{B_4} \big(\tilde{f}_\alpha \cdot D_\alpha \varphi+ \bar{A}^{\alpha 1} \tilde{f}_1 \cdot D_\alpha \varphi \big) \, dx
+\int_{B_4} \tilde{g} e_1 \cdot D_1 \varphi \,dx
\end{aligned}
\end{equation}
for any $\varphi \in W^{1,2}_0(B_4)^d$.
Let $R\in [2r,2)$ so that $B_1\setminus B_R(x_0)\neq \emptyset$, and let $\cL^*_0$ be the adjoint operator of $\cL_0$, i.e.,
$$
\cL_0^* v=D_\alpha(\bar{A}^{\alpha\beta}_*D_\beta v), \quad \bar{A}^{\alpha\beta}_*=(\bar{A}^{\beta\alpha})^{\top}.
$$
Then by \cite[Lemma 3.2]{MR3693868}, for given
$$
\phi_\alpha\in C^\infty_0((B_{2R}(x_0)\setminus B_R(x_0)) \cap B_1)^d\quad \text{and} \quad \psi\in C^\infty_0((B_{2R}(x_0)\setminus B_R(x_0)) \cap B_1),
$$
there exists a unique $(v,\pi)\in W^{1,2}_0(B_4)^d\times \tilde{L}^2(B_4)$ satisfying
\begin{equation}		\label{170904@eq3}
\left\{
\begin{aligned}
\cL_0^* v + \nabla \pi = D_\alpha \phi_\alpha & \quad \text{in }\, B_4,\\
\operatorname{div} v = \psi - (\psi)_{B_4}& \quad \text{in }\, B_4,
\end{aligned}
\right.
\end{equation}
and
\begin{equation}		\label{171124@eq5}
\|Dv\|_{L^2(B_4)}\lesssim_{d,\lambda} \||\phi_\alpha|+|\psi|\|_{L^2((B_{2R}(x_0)\setminus B_R(x_0))\cap B_1)}.
\end{equation}
By setting $\varphi =v$ in \eqref{170904@eq2} and applying $u$ as a test function to \eqref{170904@eq3}, we have
$$
\int_{B_4} (D_\alpha u \cdot  \phi_\alpha + p \psi )\,dx
 = \sum_{\alpha=2}^d \int_{B_4} \tilde{f}_\alpha \cdot D_\alpha v \,dx
+\int_{B_4} \tilde{f}_1 \cdot V\, dx
+\int_{B_4} \tilde{g} D_1 v^1 \,dx,
$$
where $V=\bar{A}^{1\beta}_*D_\beta v+\pi e_1$.
Since $\tilde{f}_\alpha\in \tilde{L}^2(B_4)^d$ and $\tilde{g}\in \tilde{L}^2(B_4)$ are supported in $B_r(x_0)$, we get
\begin{equation}		\label{170904@eq8}
\begin{aligned}
&\int_{B_4} (D_\alpha u \cdot  \phi_\alpha + p \psi )\,dx  = \sum_{\alpha=2}^d \int_{B_r(x_0)} \tilde{f}_\alpha \cdot \big(D_\alpha v - (D_\alpha v)_{B_r(x_0)}\big)\,dx\\
&\quad +\int_{B_r(x_0)} \tilde{f}_1 \cdot \big(V-(V)_{B_r(x_0)} \big)\, dx+\int_{B_r(x_0)} \tilde{g} \big(D_1 v^1-(D_1 v^1)_{B_r(x_0)}\big) \,dx.
\end{aligned}
\end{equation}
Observe that
$$
\left\{
\begin{aligned}
\cL_0^* v + \nabla \pi = 0 & \quad \text{in }\, B_R(x_0),\\
\operatorname{div} v = -(\psi)_{B_4} & \quad \text{in }\, B_R(x_0).
\end{aligned}
\right.
$$
Since
$$
D_1v^1=-\sum_{\alpha=2}^d D_\alpha v^\alpha-(\psi)_{B_4},
$$
we have
$$
[D_1 v^1]_{C^{0,1}(B_r(x_0))}\lesssim_d [D_{x'}v]_{C^{0,1}(B_r(x_0))}.
$$
Hence, it follows from \eqref{170908@eq1a} and \eqref{171124@eq5} that
\begin{align*}
&[V]_{C^{0,1}(B_{r}(x_0))}+[D_{x'}v]_{C^{0,1}(B_{r}(x_0))}+[D_1v^1]_{C^{0,1}(B_r(x_0))}\\
&\lesssim R^{-d/2-1}\|Dv\|_{L^2(B_R(x_0))}\\
&\lesssim R^{-d/2-1}\||\phi_\alpha|+|\psi|\|_{L^2((B_{2R}(x_0)\setminus B_R(x_0))\cap B_1)}.
\end{align*}
This together with \eqref{170904@eq8} yields that
\begin{equation}		\label{171108@eq1}
\begin{aligned}
&\bigg|\int_{(B_{2R}(x_0)\setminus B_R(x_0))\cap B_1} (D_\alpha u \cdot  \phi_\alpha   + p \psi ) \,dx\bigg|\\
&\lesssim
r R^{-d/2-1}
 M \||\phi_\alpha|+|\psi|\|_{L^2((B_{2R}(x_0)\setminus B_R(x_0))\cap B_1)},
\end{aligned}
\end{equation}
where
$$
M:=\int_{B_r(x_0)} (|\tilde{f}_\alpha|+|\tilde{g}|)\,dx.
$$
By the duality and H\"older's inequality, we obtain
$$
\int_{(B_{2R}(x_0)\setminus B_R(x_0))\cap B_1} (|Du|+|p|)\,dx \lesssim  \frac{r}{R} M.
$$
Let $N$ be the smallest positive integer such that $B_1\subset B_{2^N r}(x_0)$.
By taking $R=2^ir$, $i\in \{1,2,\ldots, N-1\}$, we have
$$
\int_{B_1\setminus B_{2r}(x_0)} (|Du|+|p|) \, dx \lesssim \sum_{k=1}^{N-1} 2^{-k} M \lesssim M,
$$
which implies that the map $T$ satisfies the hypothesis of Lemma \ref{170823@lem1}.
The lemma is proved.
\end{proof}

\section{Proof of Theorems \ref{M1}, \ref{M0}, \ref{M3}, and \ref{M2}}		\label{Sec4}

\subsection{Proof of Theorem \ref{M1}}		\label{171002@sec2}
We adapt the arguments in the proof of \cite[Theorem 1.5]{MR3620893}.
We shall derive a priori estimates for $(u,p)$ under the assumption that $A^{\alpha\beta}$, $f_\alpha$, and $g$ are sufficiently smooth, so that $(u,p)\in C^1(\overline{B_4})^d\times C(\overline{B_4})$.
The general case follows from a standard approximation argument (see \cite[pp. 134--135]{MR2927619}) and the $W^{1,q_0}_0$-solvability of the Stokes system (see \cite[Corollary 5.3]{MR3693868}).

Throughout the proof, we set $q=1/2$ and
$$
\Phi(x_0, r):=\inf_{\substack{\theta\in \bR^d \\ \Theta\in \bR^{d\times (d-1)}}}\bigg(\dashint_{B_{r}(x_0)}|\hat{U}-\theta|^q+|D_{x'}u-\Theta|^q\,dx\bigg)^{1/q},
$$
where $\hat{U}=A^{1\beta}D_\beta u+p e_1-f_1$.
For $\gamma\in (0,1)$, $\kappa\in (0,1/2]$, $m\in \{3,4\}$, and $f\in L^1(B_6)$, we denote
\begin{equation}		\label{171118@eq1}
\begin{aligned}
\omega_{f,x',B_m}(r)&:=\sup_{x\in B_m}\dashint_{B_r(x)}\bigg|f(y)-\dashint_{B_r'(x')}f(y_1,z')\,dz'\bigg|\,dy,\\
\tilde{\omega}_{f,x',B_m}(r)&:=\sum_{i=1}^\infty \kappa^{\gamma i} \big(\omega_{f,x',B_m}(\kappa^{-i} r )[\kappa^{-i}r<1]+\omega_{f,x',B_m}(1)[\kappa^{-i} r \ge 1]\big),
\end{aligned}
\end{equation}
where we used Iverson bracket notation, i.e., $[P]=1$ if $P$ is true and $[P]=0$ otherwise.

\begin{lemma}		\label{171020@lem1z}
Let $\gamma\in (0,1)$.
Under the same hypothesis of Theorem \ref{M1} $(a)$,
there exists a constant $\kappa\in (0,1/2]$ depending only on $d$, $\lambda$, and $\gamma$, such that  the following hold.
\begin{enumerate}[$(i)$]
\item
For any $x_0\in B_3$ and $0<r\le 1/4$, we have
\begin{equation}		\label{171023@eq9}
\begin{aligned}
\sum_{j=0}^\infty \Phi(x_0, \kappa^j r) &\lesssim_{d,\lambda, \gamma} \Phi(x_0,r)+\|Du\|_{L^\infty(B_r(x_0))}\int_0^r \frac{\tilde{\omega}_{A^{\alpha\beta},x',B_4}(t)}{t}\,dt\\
&\quad +\int_0^r \frac{\tilde{\omega}_{f_\alpha,x',B_4}(t)+\tilde{\omega}_{g,x',B_4}(t)}{t}\,dt,
\end{aligned}
\end{equation}
where each integration is finite; see Lemma \ref{171118@lem1} $(a)$.
\item
For any $x_0\in B_3$ and $0<\rho\le r \le  1/4$, we have
\begin{equation}	\label{171023@1a}
\begin{aligned}
\Phi(x_0,\rho) &\lesssim_{d,\lambda,\gamma} \left(\frac{\rho}{r}\right)^\gamma \Phi(x_0,r)+ \|Du\|_{L^\infty(B_r(x_0))}\tilde{\omega}_{A^{\alpha\beta},x',B_3}(\rho)\\
&\quad +\tilde{\omega}_{f_\alpha,x',B_3}(\rho)+\tilde{\omega}_{g,x',B_3}(\rho).
\end{aligned}
\end{equation}
\end{enumerate}
\end{lemma}

\begin{proof}
Fix $x_0=(x_{01},x_0')\in B_3$ and $0<r \le 1/4$.
For a given function $f$, we denote
$$
\bar{f}(x_1):=\dashint_{B_r'(x_0')}f(x_1,y')\,dy'.
$$
Define an elliptic operator $\cL_0$ by
$$
\cL_0 u=D_\alpha(\bar{A}^{\alpha\beta}(x_1)D_\beta u).
$$
Let $u_0=(u_0^1(x_1),\ldots,u_0^d(x_1))^{\top}$ and $p_0=p_0(x_1)$ be functions satisfying
$$
u_0^1=\bar{g}, \quad \bar{A}^{11} u_0+p_0 e_1= \bar{f}_1.
$$
Then the pair $(u_e, p_e)$ given by
$$
u_e=u-\int_{-6}^{x_1} u_0 \,d y_1 \quad \text{and} \quad p_e = p-p_0
$$
satisfies
$$
\left\{
\begin{aligned}
\cL_0 u_e+\nabla p_e=D_\alpha F_\alpha \quad \text{in }\, B_6, \\
\operatorname{div} u_e=G \quad \text{in }\, B_6,
\end{aligned}
\right.
$$
where $F_\alpha=(\bar{A}^{\alpha\beta}-A^{\alpha\beta})D_\beta u+f_\alpha-\bar{f}_\alpha$ and $G=g-\bar{g}$.
We decompose
\begin{equation}		\label{171023_D3}
(u_e,p_e)=(w,p_1)+(v,p_2),
\end{equation}
where $(w, p_1)\in W^{1,2}_0(B_{4r}(x_0))^d\times \tilde{L}^2(B_{4r}(x_0))$ is the weak solution of the problem
$$
\left\{
\begin{aligned}
\cL_0 w+\nabla p_1=D_\alpha \big(I_{B_{r}(x_0)} F_\alpha \big) \quad \text{in }\, B_{4r}(x_0),\\
\operatorname{div} w=I_{B_{r}(x_0)} G-\big(I_{B_{r}(x_0)} G\big)_{B_{4r}(x_0)} \quad \text{in }\, B_{4r}(x_0).
\end{aligned}
\right.
$$
Here, $I_{B_r(x_0)}$ is the characteristic function.
By Lemma \ref{170904@lem1} with scaling, we have for $t>0$ that
$$
\big|\{x\in B_{r}(x_0): |Dw(x)|+|p_1(x)|>t\}\big| \lesssim_{d,\lambda} \frac{1}{t} \int_{B_{r}(x_0)} |F_\alpha|+|G| \, dx.
$$
Then we obtain for $\tau>0$ that
\begin{align*}
&\int_{B_{r}(x_0)}(|Dw|+|p_1|)^q \, dx \\
&=\bigg(\int_0^\tau+\int_\tau ^\infty \bigg) q t^{q-1} \big|\{x\in B_{r}(x_0): |Dw(x)|+|p_1(x)|>t\}\big| \, dt\\
&\lesssim  |B_{r}| \tau^q +  \bigg(\int_{B_{r}(x_0)} |F_\alpha|+|G| \,dx\bigg) \tau^{q-1}.
\end{align*}
By optimizing over $\tau$ and taking the $q$-th root, we have
\begin{equation}		\label{170913@eq2}
\bigg(\dashint_{B_{r}(x_0)}(|Dw|+|p_1|)^q \, dx\bigg)^{1/q} \lesssim \dashint_{B_{r}(x_0)} |F_\alpha|+|G|\,dx.
\end{equation}
Observe that $(v,p_2)=(u_e,p_e)-(w,p_1)$ satisfies
$$
\left\{
\begin{aligned}
\cL_0 v+\nabla p_2=0 & \quad \text{in }\, B_r(x_0),\\
\operatorname{div} v= \big(I_{B_r(x_0)}G\big)_{B_{4r}(x_0)} & \quad \text{in }\, B_r(x_0).
\end{aligned}
\right.
$$
Then by \eqref{171110@eq1}, we obtain for $\kappa\in (0,1/2]$ that
\begin{equation}		\label{171030@eq8}
\begin{aligned}
&\bigg(\dashint_{B_{\kappa r}(x_0)}|V-(V)_{B_{\kappa r}(x_0)}|^q +|D_{x'}v -(D_{x'}v)_{B_{\kappa r}(x_0)}|^q \, dx\bigg)^{1/q}\\
&\lesssim_{d,\lambda}
\kappa
\inf_{\substack{\theta\in \bR^d \\ \Theta\in \bR^{d\times (d-1)}}} \bigg(\dashint_{B_r(x_0)}|V-\theta |^q +|D_{x'}v -\Theta|^q \, dx\bigg)^{1/q},
\end{aligned}
\end{equation}
where $V=\bar{A}^{1\beta}D_\beta v+p_2 e_1$.
Notice from the decomposition \eqref{171023_D3} that
$$
\begin{aligned}
& \bigg(\dashint_{B_{\kappa r}(x_0)}|U_e-(V)_{B_{\kappa r}(x_0)}|^q+|D_{x'}u_e-(D_{x'}v)_{B_{\kappa r}(x_0)}|^q\,dx\bigg)^{1/q}\\
&\lesssim_{d,\lambda} \bigg(\dashint_{B_{\kappa r}(x_0)}|V-(V)_{B_{\kappa r}(x_0)}|^q+|D_{x'}v-(D_{x'}v)_{B_{\kappa r}(x_0)}|^q\,dx\bigg)^{1/q}\\
&\quad + \bigg(\dashint_{B_{\kappa r}(x_0)} (|Dw|+|p_1|)^q\,dx\bigg)^{1/q},
\end{aligned}
$$
where $U_e=\bar{A}^{1\beta} D_\beta u_e+p_e e_1$.
Using this together with \eqref{170913@eq2} and \eqref{171030@eq8}, we have
$$
\begin{aligned}
& \bigg(\dashint_{B_{\kappa r}(x_0)}|U_e-(V)_{B_{\kappa r}(x_0)}|^q+|D_{x'}u_e-(D_{x'}v)_{B_{\kappa r}(x_0)}|^q\,dx\bigg)^{1/q}\\
&\lesssim_{d,\lambda}
\kappa  \inf_{\substack{\theta\in \bR^d \\ \Theta\in \bR^{d\times (d-1)}}}\bigg(\dashint_{B_{r}(x_0)}|U_e-\theta|^q+|D_{x'}u_e-\Theta|^q\,dx\bigg)^{1/q}\\
&\quad + \kappa^{-d/q}\dashint_{B_r(x_0)}|F_\alpha|+|G|\,dx.
\end{aligned}
$$
Therefore, from the fact that
$$
D_{x'} u_e=D_{x'} u, \quad |\hat{U}-U_e|\lesssim |A^{1\beta}-\bar{A}^{1\beta}|| D_\beta u | +|f_1-\bar{f}_1|,
$$
we get
$$
\begin{aligned}
\Phi(x_0,\kappa r) &\le C_0
\kappa \Phi(x_0, r) + C_0 \kappa^{-d/q}\|Du\|_{L^\infty(B_{r}(x_0))} \omega_{A^{\alpha\beta},x',B_3}(r)\\
&\quad  +C_0\kappa^{-d/q} \big(\omega_{f_\alpha,x',B_3}(r)+\omega_{g,x',B_3}(r)\big),
\end{aligned}
$$
where $C_0=C_0(d,\lambda)>0$.
Taking $\kappa=\kappa(d,\lambda,\gamma)\in (0,1/2]$ so that
$C_0 \kappa^{1-\gamma}\le 1$,
we have
\begin{equation}		\label{171023@eq3}
\begin{aligned}
\Phi(x_0,\kappa r) &\le \kappa^{\gamma} \Phi(x_0, r)+ C\|Du\|_{L^\infty(B_{r}(x_0))} \omega_{A^{\alpha\beta},x',B_3}(r)\\
&\quad +C\big(\omega_{f_\alpha,x',B_3}(r)+\omega_{g,x',B_3}(r)\big),
\end{aligned}
\end{equation}
where $C=C(d,\lambda,\gamma)$.
By iterating, we see that
\begin{equation}		\label{170926@eq1}
\begin{aligned}
\Phi(x_0, \kappa^j r) & \le \kappa^{\gamma j}\Phi(x_0,r) +C\|Du\|_{L^\infty(B_{ r}(x_0))} \tilde{\omega}_{A^{\alpha\beta},x',B_3}(\kappa^j r)\\
&\quad +C\big(\tilde{\omega}_{f_\alpha,x',B_3}(\kappa^j r)+\tilde{\omega}_{g,x',B_3}(\kappa^j r)\big),
\end{aligned}
\end{equation}
where we used the fact that
$$
\sum_{i=1}^j \kappa^{\gamma(i-1)} \omega_{\bullet,x',B_3}(\kappa^{j-i}r)
\le
\kappa^{-\gamma} \tilde{\omega}_{\bullet,x',B_3}(\kappa^j r).
$$
Taking the summations of both sides of \eqref{170926@eq1} with respect to $j=0,1,2,\ldots$, and using Lemma \ref{171118@lem1} $(a)$, we get \eqref{171023@eq9}.

For any $\rho\in (0, r]$, we take an integer $j$ such that
$$
\kappa^{j+1}<\frac{\rho}{r}\le \kappa^j.
$$
If $j=0$, then obviously we have
$$
\Phi(x_0,\rho)\lesssim_{d,\kappa}\Phi(x_0,r)\lesssim_{d,\kappa,\gamma} \left(\frac{\rho}{r}\right)^\gamma \Phi(x_0,r),
$$
which implies \eqref{171023@1a}.
On the other hand, if $j\ge 1$,
then by using \eqref{170926@eq1} with $\rho$ in place of $\kappa^j r$, we obtain
$$
\begin{aligned}
\Phi(x_0, \rho) &\lesssim \left(\frac{\rho}{r}\right)^\gamma \Phi(x_0,r)  + \|Du\|_{L^\infty(B_{r}(x_0))} \tilde{\omega}_{A^{\alpha\beta},x',B_3}(\rho)\\
&\quad +\tilde{\omega}_{f_\alpha,x',B_3}(\rho)+\tilde{\omega}_{g,x',B_3}(\rho).
\end{aligned}
$$
This gives \eqref{171023@1a}.
The lemma is proved.
\end{proof}

Now we are ready to prove the assertion (a) in the theorem.

\begin{proof}[Proof of Theorem \ref{M1} $(a)$]
In this proof, we fix $\gamma \in (0,1)$, and let $\kappa=\kappa(d,\lambda,\gamma)$ be the constant from Lemma \ref{171020@lem1z}.
We denote
$$
\cU=|\hat{U}|+|D_{x'}u|, \quad \tilde{\omega}_{\bullet,x'}=\tilde{\omega}_{\bullet,x',B_4}, \quad \cF(r)=\int_0^r \frac{\tilde{\omega}_{f_{\alpha},x'}(t)+\tilde{\omega}_{g,x'}(t)}{t}\,dt.
$$

We first derive $L^\infty$-estimates for $Du$ and $p$.
Let $x_0\in B_3$ and $0<r \le 1/4$.
We take $\theta_{x_0,r}\in \bR^d$ and $\Theta_{x_0, r}\in \bR^{d\times (d-1)}$ to be such that
$$
\Phi(x_0, r)=\bigg(\dashint_{B_r(x_0)} \big|\hat{U}-\theta_{x_0,r}\big|^q+|D_{x'}u-\Theta_{x_0,r}\big|^q\,dx\bigg)^{1/q}.
$$
Similarly, we find $\theta_{x_0,\kappa^i r}\in \bR^d$ and $\Theta_{x_0, \kappa^i r}\in \bR^{d\times (d-1)}$ for $i\in \{1,2,\ldots\}$.
Recall the assumption that $A^{\alpha\beta}$ and $f_\alpha$ are sufficiently smooth, so that $(u,p)\in C^1(\overline{B_4})^d\times C(\overline{B_4})$.
Thus, since the right-hand side of \eqref{170926@eq1} goes to zero as $j\to \infty$, we see that
\begin{equation}		\label{171031@eq1}
\lim_{i\to \infty} \theta_{x_0,\kappa^i r}=\hat{U}(x_0), \quad \lim_{i\to \infty}\Theta_{x_0, \kappa^i r}=D_{x'}u(x_0).
\end{equation}
By averaging the inequality
$$
|\theta_{x_0,\kappa r}-\theta_{x_0,r}|^q\le |\hat{U}-\theta_{x_0,\kappa r}|^q+|\hat{U}-\theta_{x_0,r}|^q
$$
on $B_{\kappa r}(x_0)$ and taking the $q$-th root, we have
$$
|\theta_{x_0,\kappa r}-\theta_{x_0,r}|\lesssim \Phi(x_0,\kappa r)+\Phi(x_0, r).
$$
We apply the above inequality iteratively and use \eqref{171031@eq1} to get
\begin{equation}		\label{171120@eq2}
|\hat{U}(x_0)-\theta_{x_0,r}|\lesssim \sum_{j=0}^\infty \Phi(x_0, \kappa^j r).
\end{equation}
Averaging the inequality
$$
|\theta_{x_0,r}|^q\le |\hat{U}-\theta_{x_0,r}|^q+|\hat{U}|^q
$$
on $B_{r}(x_0)$ and taking the $q$-th root, we obtain that
$$
|\theta_{x_0,r}|\lesssim \Phi(x_0,r)+r^{-d}\|\hat{U}\|_{L^1(B_r(x_0))}\lesssim r^{-d}\|\cU\|_{L^1(B_r(x_0))}.
$$
By Lemma \ref{171020@lem1z} $(i)$, \eqref{171120@eq2}, and the above inequality, we have
$$
|\hat{U}(x_0)|\lesssim r^{-d}\|\cU\|_{L^1(B_r(x_0))} +\|Du\|_{L^\infty(B_r(x_0))} \int_0^r \frac{\tilde{\omega}_{A^{\alpha\beta},x'}(t)}{t}\,dt+\cF(r).
$$
Since the same inequality holds for $|D_{x'}u(x_0)|$, we have
\begin{equation}		\label{171031@eq1c}
\cU(x_0)\lesssim r^{-d}\|\cU\|_{L^1(B_r(x_0))}+\|Du\|_{L^\infty(B_r(x_0))} \int_0^r \frac{\tilde{\omega}_{A^{\alpha\beta},x'}(t)}{t}\,dt+\cF(r).
\end{equation}
Notice from the definition of $\hat{U}$ that
$$
\sum_{j=2}^{d} {A}^{11}_{ij} D_1 u^j=\hat{U}^i - \sum_{j=1}^d \sum_{\beta=2}^d {A}^{1\beta}_{ij}D_\beta u^j-{A}^{11}_{i1}D_1 u^1+f_1^i, \quad i\in \{2,\ldots,d\}.
$$
By the ellipticity condition on ${A}^{\alpha\beta}$, $[{A}^{11}_{ij}]_{i, j=2}^d$ is nondegenerate.
This together with
\begin{equation}		\label{171207@eq1}
D_1u^1=g-\sum_{i=2}^d D_i u^i
\end{equation}
gives
$$
|Du|\le |D_1 u^1 | + \sum_{j=2}^d |D_1 u^j |+|D_{x'}u|\lesssim \cU+|f_1|+|g|.
$$
Therefore, we obtain by \eqref{171031@eq1c} that
\begin{equation}		\label{171024@eq5}
\begin{aligned}
\cU(x_0) &\le C_1 r^{-d}\|\cU \|_{L^1(B_r(x_0))}+C_1\|\cU\|_{L^\infty(B_r(x_0))} \int_0^r \frac{\tilde{\omega}_{A^{\alpha\beta},x'}(t)}{t}\,dt\\
&\quad +C_1\||f_1|+|g|\|_{L^\infty(B_r(x_0))}\int_0^r \frac{\tilde{\omega}_{A^{\alpha\beta},x'}(t)}{t}\,dt + C_1\cF(r),
\end{aligned}
\end{equation}
where $C_1=C_1(d,\lambda,\gamma)$.
Taking $r_0 \in (0, 1/4]$ so that
\begin{equation}		\label{180305@eq1}
C_1\int_0^{r_0} \frac{\tilde{\omega}_{A^{\alpha\beta},x'}(t)}{t}\,dt\le \frac{1}{3^d},
\end{equation}
we have for $x_0\in B_3$ and $0<r\le r_0$ that
\begin{equation}		\label{171127@eq1}
\begin{aligned}
\cU(x_0) &\le C_1 r^{-d}\|\cU \|_{L^1(B_r(x_0))}+3^{-d}\|\cU\|_{L^\infty(B_r(x_0))} \\
&\quad +3^{-d}\||f_1|+|g|\|_{L^\infty(B_r(x_0))} + C_1\cF(r).
\end{aligned}
\end{equation}
Here, the constant $r_0$ depends only on $d$, $\lambda$, $\gamma$, and $\omega_{A^{\alpha\beta},x'}$.

For $k\in \{1,2,\ldots\}$, we denote $r_k=3-2^{1-k}$.
Since $r_{k+1}-r_k=2^{-k}$, we have $B_{r}(x_0)\subset B_{r_{k+1}}$ for any $x_0\in B_{r_k}$ and $r= 2^{-k}$.
We take $k_0$ sufficiently large such that $2^{-k_0}\le r_0$.
Then by \eqref{171127@eq1} with $r=2^{-k}$, we have for $k\ge k_0$ that
$$
\begin{aligned}
 \|\cU\|_{L^\infty(B_{r_k})} &\le C_1 2^{dk}\|\cU\|_{L^1(B_{6})}+3^{-d}\|\cU\|_{L^\infty(B_{r_{k+1}})}\\
&\quad +3^{-d}\||f_1|+|g|\|_{L^\infty(B_{6})}+C_1\cF(1).
\end{aligned}
$$
By multiplying both sides of the above inequality by $3^{-dk}$ and summing the terms with respect to $k=k_0,k_0+1,\ldots$, we see that
$$
\begin{aligned}
\sum_{k=k_0}^\infty 3^{-dk} \|\cU\|_{L^\infty(B_{r_k})} &\le C \|\cU\|_{L^1(B_{6})}+\sum_{k=k_0+1}^\infty 3^{-dk}\|\cU\|_{L^\infty(B_{r_{k}})}\\
&\quad +C \||f_1|+|g|\|_{L^\infty(B_{6})}+C\cF(1),
\end{aligned}
$$
where each summation is finite and $C=C(d,\lambda, \gamma)$.
By subtracting
$$
\sum_{k=k_0+1}^\infty 3^{-dk}\|\cU\|_{L^\infty(B_{r_k})}
$$
from both sides of the above inequality, we have
$$
\|\cU\|_{L^\infty(B_2)} \lesssim_{d,\lambda,\gamma,\omega_{A^{\alpha\beta},x'}} \|\cU\|_{L^1(B_6)}+\||f_1|+|g|\|_{L^\infty(B_{6})}+\cF(1).
$$
Therefore, using \eqref{171207@eq1} and the fact that
$$
|p|\lesssim \cU +|Du|+|f_1|,
$$
we get the following $L^\infty$-estimate for $Du$ and $p$:
\begin{equation}		\label{170925_eq1}
\begin{aligned}
&\|Du\|_{L^\infty(B_2)}+\|p\|_{L^\infty(B_2)} \\
& \le C\big( \|Du\|_{L^1(B_6)}+\|p\|_{L^1(B_6)}+\|f_1\|_{L^\infty(B_6)}+\|g\|_{L^\infty(B_{6})}+\cF(1)\big),
\end{aligned}
\end{equation}
where $C=C(d,\lambda, \gamma, \omega_{A^{\alpha\beta},x'})$.

\begin{remark}		\label{180305@rmk1}
With regard to the dependency of the constant $C$ in \eqref{170925_eq1}, the parameter $\omega_{A^{\alpha\beta},x'}$ can be replaced by a function $\omega_0:(0,1]\to [0, \infty)$ such that
$$
\int_0^r\frac{\tilde{\omega}_{A^{\alpha\beta},x'}(t)}{t}\,dt\le \int_0^r\frac{\omega_0(t)}{t}\,dt<\infty \quad \text{for all }\, r\in (0,1].
$$
Indeed, this can be verified by taking $r_0\in (0,1/4]$ so that
$$
C_1\int_0^{r_0} \frac{\omega_0(t)}{t}\,dt\le \frac{1}{3^d}
$$
in \eqref{180305@eq1}.
\end{remark}

Next, we shall derive estimates of the modulus of continuity of $\hat{U}$ and $D_{x'}u$.
By Lemma \ref{171020@lem1z} $(ii)$, we have for $x_0\in B_1$ and $0<\rho\le 1/4$ that
$$
\Phi(x_0, \rho)\lesssim \rho^{\gamma}\|\cU\|_{L^1(B_6)}+\|Du\|_{L^\infty(B_2)}\int_0^\rho \frac{\tilde{\omega}_{A^{\alpha\beta},x'}(t)}{t}\,dt+\cF(\rho),
$$
where we used the fact that
$$
\Phi(x_0, 1/4)\lesssim \|\cU\|_{L^1(B_6)}, \quad \tilde{\omega}_{\bullet,x',B_3}(\rho) \lesssim \int_0^\rho \frac{\tilde{\omega}_{\bullet,x'}(t)}{t}\,dt.
$$
From this together with Lemma \ref{171020@lem1z} $(i)$, we get
\begin{equation}		\label{171120@A2}
\sum_{j=0}^\infty \Phi(x_0,\kappa^j \rho)\lesssim \rho^{\gamma}\|\cU\|_{L^1(B_6)}+\|Du\|_{L^\infty(B_2)}\int_0^\rho\frac{\tilde{\omega}_{A^{\alpha\beta},x'}(t)}{t}\,dt+\cF(\rho).
\end{equation}
Let $x,y\in {B_1}$ with $\rho:=|x-y|\le 1/4$.
Then for any $z\in B_\rho(x)\cap B_\rho(y)$, we have
\begin{align*}
|\hat{U}(x)-\hat{U}(y)|^q &\le |\hat{U}(x)-\theta_{x,\rho}|^q
+ |\theta_{x,\rho}-\theta_{y,\rho}|^q
+ |\hat{U}(y)-\theta_{y,\rho}|^q\\
&\le 2\sup_{x_0\in {B_1}} |\hat{U}(x_0)-\theta_{x_0,\rho} |^q
+|\hat{U}(z)-\theta_{x,\rho}|^q+|\hat{U}(z)-\theta_{y,\rho}|^q.
\end{align*}
By taking average over $z\in B_\rho(x)\cap B_\rho(y)$ and taking the $q$-th root, we see that
\begin{align*}
|\hat{U}(x)-\hat{U}(y)|&\lesssim \sup_{x_0\in {B_1}} |\hat{U}(x_0)-\theta_{x_0,\rho} |
+\Phi(x,\rho)+\Phi(y,\rho)\\
&\lesssim \sup_{x_0\in B_1}\Bigg(\sum_{j=0}^\infty \Phi(x_0, \kappa^j \rho)+\Phi(x_0,\rho)\Bigg)\lesssim \sup_{x_0\in B_1}\sum_{j=0}^\infty \Phi(x_0, \kappa^j \rho),
\end{align*}
where we used \eqref{171120@eq2} in the second inequality.
Hence we get from \eqref{171120@A2} that
\begin{align*}
|\hat{U}(x)-\hat{U}(y)| &\lesssim |x-y|^{\gamma} \|\cU\|_{L^1(B_6)}\\
&\quad + \|Du\|_{L^\infty(B_2)}\int_0^{|x-y|} \frac{\tilde{\omega}_{A^{\alpha\beta},x'}(t)}{t}\,dt +\cF(|x-y|).
\end{align*}
Similarly, we have the same estimate for $D_{x'}u$.
Thus using \eqref{170925_eq1}, we conclude that
\begin{equation}		\label{171128@A1}
\begin{aligned}
&|\hat{U}(x)-\hat{U}(y)|+|D_{x'}u(x)-D_{x'}u(y)| \\
&\le C |x-y|^{\gamma} \big(\|Du\|_{L^1(B_6)}+\|p\|_{L^1(B_6)}+\|f_1\|_{L^1(B_6)}\big)\\
&\quad + C\big(\|Du\|_{L^1(B_6)}+\|p\|_{L^1(B_6)}\big)\int_0^{|x-y|} \frac{\tilde{\omega}_{A^{\alpha\beta},x'}(t)}{t}\,dt \\
&\quad +C\big(\|f_1\|_{L^\infty(B_6)}+\|g\|_{L^\infty(B_6)}\big)\int_0^{|x-y|} \frac{\tilde{\omega}_{A^{\alpha\beta},x'}(t)}{t}\,dt\\
&\quad +C\cF(1) \int_0^{|x-y|} \frac{\tilde{\omega}_{A^{\alpha\beta},x'}(t)}{t}\,dt+C\cF(|x-y|)
\end{aligned}
\end{equation}
for any $x,y\in B_1$ with $|x-y|\le 1/4$,
where $C>0$ is a constant depending only on $d$, $\lambda$, $\gamma$, and $\omega_{A^{\alpha\beta},x'}$.
We note that
if $x,y\in B_1$ with $|x-y|>1/4$, then by \eqref{170925_eq1}, we have
\begin{align}
\nonumber
&|\hat{U}(x)-\hat{U}(y)|+|D_{x'}u(x)-D_{x'}u(y)| \\
\nonumber
&\le C \|Du\|_{L^\infty(B_1)}+\|p\|_{L^\infty(B_1)}+\|f_1\|_{L^\infty(B_1)}\\
\label{171219@eq2}
&\le C |x-y|^\gamma \big(\|Du\|_{L^1(B_6)}+\|p\|_{L^1(B_6)}+\|f_1\|_{L^\infty(B_6)}+\|g\|_{L^\infty(B_6)}+\cF(1)\big).
\end{align}
The assertion $(a)$ in Theorem \ref{M1} is proved.
\end{proof}

We now turn to the proof of the assertion (b) in the theorem.

\begin{proof}[Proof Theorem \ref{M1} $(b)$]
In this proof, we set $\gamma=\frac{1+\gamma_0}{2}\in (0,1)$.
Let $\kappa=\kappa(d,\lambda,\gamma)$ be the constant from Lemma \ref{171020@lem1z}, and recall the notation
$$
\tilde{\omega}_{\bullet,x',B_4}(r)=\sum_{i=1}^\infty \kappa^{\gamma i} \big(\omega_{\bullet,x',B_4}(\kappa^{-i} r )[\kappa^{-i}r<1]+\omega_{\bullet,x',B_4}(1)[\kappa^{-i} r \ge 1]\big).
$$
Notice from Lemma \ref{171118@lem1} $(b)$ that for any function $f$ satisfying $[f]_{C^{\gamma_0}_{x'}(B_6)}<\infty$, we have
$$
\tilde{\omega}_{f,x',B_4}(r)+\int_0^r \frac{\tilde{\omega}_{f,x',B_4}(t)}{t}\,dt \lesssim [f]_{C^{\gamma_0}_{x'}(B_6)}r^{\gamma_0}.
$$
Therefore, by \eqref{170925_eq1} together with Remark \ref{180305@rmk1}, we have the following $L^\infty$-estimate for $Du$ and $p$:
$$
\begin{aligned}
&\|Du\|_{L^\infty(B_2)}+\|p\|_{L^\infty(B_2)}  \lesssim_{d,\lambda,\gamma_0,[A^{\alpha\beta}]_{C^{{\gamma_0}}_{x'}(B_6)}} \|Du\|_{L^1(B_6)}+\|p\|_{L^1(B_6)}\\
&\quad +\|f_1\|_{L^\infty(B_6)}+\|g\|_{L^\infty(B_{6})}+[f_\alpha]_{C^{{\gamma_0}}_{x'}(B_6)}+[g]_{C^{\gamma_0}_{x'}(B_6)}.
\end{aligned}
$$
Moreover, by \eqref{171128@A1} and \eqref{171219@eq2}, we have the following $C^{\gamma_0}$-estimate for $\hat{U}$ and $D_{x'}u$:
\begin{align*}
&[\hat{U}]_{C^{\gamma_0}(B_1)}+[D_{x'}u]_{C^{\gamma_0}(B_1)} \lesssim_{d,\lambda,{\gamma_0},[A^{\alpha\beta}]_{C^{{\gamma_0}}_{x'}(B_6)}} \|Du\|_{L^1(B_6)}+\|p\|_{L^1(B_6)}\\
&\quad +\|f_1\|_{L^\infty(B_6)}+\|g\|_{L^\infty(B_6)}+[f_\alpha]_{C^{\gamma_0}_{x'}(B_6)}+[g]_{C^{\gamma_0}_{x'}(B_6)}.
\end{align*}
This completes the proof of the assertion $(b)$ in Theorem \ref{M1} and that of Theorem \ref{M1}.
\end{proof}

\subsection{Proof of Theorem \ref{M0}}		\label{170929@sec1}
We shall derive a priori estimates for $(u,p)$ by assuming $(u,p)\in C^1(\overline{B_4})^d\times C(\overline{B_4})$.
We again set $q=1/2$ and
$$
\Phi(x_0,r):=\inf_{\substack{ \theta\in \bR \\ \Theta\in \bR^{d\times d}}}\bigg(\dashint_{B_r(x_0)}|Du-\Theta|^q+|p-\theta |^q\,dx\bigg)^{1/q}.
$$
For $\gamma\in (0,1)$, $\kappa\in (0,1/2]$, $m\in \{3,4\}$, and $f\in L^1(B_6)$, we denote
\begin{equation}		\label{171217@eq1}
\begin{aligned}
\omega_{f,B_m}(r)&:=\sup_{x\in B_m}\dashint_{B_r(x)}\big|f(y)-(f)_{B_r(x)}\big|\,dy,\\
\tilde{\omega}_{f,B_m}(r)&:=\sum_{i=1}^\infty \kappa^{\gamma i} \big(\omega_{f,B_m}(\kappa^{-i} r )[\kappa^{-i}r<1]+\omega_{f,B_m}(1)[\kappa^{-i} r \ge 1]\big).
\end{aligned}
\end{equation}

\begin{lemma}		\label{171020@lem10z}
Let $\gamma\in (0,1)$.
Under the same hypothesis of Theorem \ref{M0} $(a)$,
there exists a constant $\kappa\in (0,1/2]$ depending only on $d$, $\lambda$, and $\gamma$, such that the following hold.
\begin{enumerate}[$(i)$]
\item
For any $x_0\in B_3$ and $0<r \le 1/4$, we have
$$
\begin{aligned}
&\sum_{j=0}^\infty \Phi(x_0,\kappa^j r) \lesssim_{d,\lambda,\gamma} \Phi(x_0,r)\\
&\quad +\|Du\|_{L^\infty(B_r(x_0))}\int_0^r \frac{\tilde{\omega}_{A^{\alpha\beta},B_4}(t)}{t}\,dt +\int_0^r \frac{\tilde{\omega}_{f_\alpha,B_4}(t)+\tilde{\omega}_{g,B_4}(t)}{t}\,dt.
\end{aligned}
$$
\item
For any $x_0\in B_3$, $0<\rho\le r \le 1/4$, we have
$$
\begin{aligned}
&\Phi(x_0,\rho)\lesssim_{d,\lambda,\gamma}\left(\frac{\rho}{r}\right)^\gamma \Phi(x_0,r)\\
&\quad + \|Du\|_{L^\infty(B_r(x_0))}\tilde{\omega}_{A^{\alpha\beta}, B_3}(\rho)+\tilde{\omega}_{f_\alpha,B_3}(\rho)+\tilde{\omega}_{g,B_3}(\rho).
\end{aligned}
$$
\end{enumerate}
\end{lemma}

\begin{proof}
The proof is similar to that of Lemma \ref{171020@lem1z}.
Let $x_0\in B_3$ and $0<r \le 1/4$.
For a function $f$, we denote
$\bar{f}=(f)_{B_r(x_0)}$.
Define an elliptic operator $\cL_0$ by
$$
\cL_0 u=D_\alpha(\bar{A}^{\alpha\beta}D_\beta u),
$$
and observe that $(u,p)$ satisfies
$$
\left\{
\begin{aligned}
\cL_0 u+\nabla p=D_\alpha F_\alpha &\quad \text{in }\, B_6,\\
\operatorname{div} u=G+\bar{g} &\quad \text{in }\, B_6,
\end{aligned}
\right.
$$
where $F_\alpha=(\bar{A}^{\alpha\beta}-A^{\alpha\beta})D_\beta u+f_\alpha -\bar{f}_\alpha$ and $G=g-\bar{g}$.
We decompose
$$
(u,p)=(w,p_1)+(v,p_2),
$$
where $(w,p_1)\in W^{1,2}_0(B_{4r}(x_0))^d\times \tilde{L}^2(B_{4r}(x_0))$ is the weak solution of the problem
$$
\left\{
\begin{aligned}
\cL_0 w+\nabla p_1=D_\alpha(I_{B_r(x_0)}F_\alpha) &\quad \text{in }\, B_{4r}(x_0),\\
\operatorname{div} w=I_{B_r(x_0)}G-\big(I_{B_r(x_0)}G\big)_{B_{4r}(x_0)} &\quad \text{in }\, B_{4r}(x_0).
\end{aligned}
\right.
$$
Then similar to \eqref{170913@eq2}, we get
\begin{equation}		\label{171020@eq5d}
\bigg(\dashint_{B_r(x_0)}(|Dw|+|p_1|)^q\,dx\bigg)^{1/q}\lesssim \dashint_{B_r(x_0)} |F_\alpha|+|G|\,dx.
\end{equation}
Observe that $(v,p_2)=(u,p)-(w,p_1)$ satisfies
$$
\left\{
\begin{aligned}
\cL_0 v+\nabla p_2=0 &\quad \text{in }\, B_r(x_0),\\
\operatorname{div} v=\big(I_{B_r(x_0)}G\big)_{B_{4r}(x_0)} + \bar{g} &\quad \text{in }\, B_r(x_0).
\end{aligned}
\right.
$$
Then by \eqref{171110@eq1a}, we obtain for $\kappa\in (0,1/2]$ that
\begin{equation}		\label{171030@eq9}
\begin{aligned}
&\bigg(\dashint_{B_{\kappa r}(x_0)}\big|Dv-(Dv)_{B_{\kappa r}(x_0)}\big|^q+\big|p_2-(p_2)_{B_{\kappa r}(x_0)}\big|^q\,dx\bigg)^{1/q}\\
&\lesssim_{d,\lambda}
\kappa  \inf_{\Theta\in \bR^{d\times d}} \bigg(\dashint_{B_r(x_0)} \big|Dv-\Theta \big|^q\,dx\bigg)^{1/q}.
\end{aligned}
\end{equation}
Observe that
$$
\begin{aligned}
\Phi(x_0,\kappa r)
&\le \bigg(\dashint_{B_{\kappa r}(x_0)}\big|Du-(Dv)_{B_{\kappa r}(x_0)}\big|^q+|p-(p_2)_{B_{\kappa r}(x_0)}\big|^q\,dx\bigg)^{1/q}\\
&\lesssim \bigg(\dashint_{B_{\kappa r}(x_0)}\big|Dv-(Dv)_{B_{\kappa r}(x_0)}\big|^q+|p_2-(p_2)_{B_{\kappa r}(x_0)}\big|^q\,dx\bigg)^{1/q}\\
&\quad +\bigg(\dashint_{B_{\kappa r}(x_0)}|Dw|^q+|p_1|^q\,dx\bigg)^{1/q}.
\end{aligned}
$$
Using this together with \eqref{171020@eq5d} and \eqref{171030@eq9}, we get
$$
\begin{aligned}
\Phi(x_0,\kappa r)&\le C_0
\kappa
\Phi(x_0,r) +C_0\kappa^{-d/q}\|Du\|_{L^\infty(B_r(x_0))}\omega_{A^{\alpha\beta}, B_3}(r)\\
&\quad +C_0\kappa^{-d/q}\big(\omega_{f_\alpha,B_3}(r)+\omega_{g,B_3}(r)\big),
\end{aligned}
$$
where $C_0=C_0(d,\lambda)>0$.
We take $\kappa=\kappa(d,\lambda,\gamma)\in (0,1/2]$ so that
$C_0\kappa^{1-\gamma}\le 1$.
Then we have
$$
\begin{aligned}
\Phi(x_0,\kappa r) &\le \kappa^\gamma \Phi(x_0,r)+C\|Du\|_{L^\infty(B_r(x_0))}\omega_{A^{\alpha\beta}, B_3}(r)\\
&\quad +C\big(\omega_{f_\alpha,B_3}(r)+\omega_{g,B_3}(r)\big),
\end{aligned}
$$
where $C=C(d,\lambda,\gamma)$,
which corresponds to \eqref{171023@eq3}.
The rest of the proof is identical to that of Lemma \ref{171020@lem1z} and omitted.
\end{proof}

Now we are ready to prove the theorem.

\begin{proof}[Proof of Theorem \ref{M0}]
The proof is an adaptation of that of Theorem \ref{M1}.
To show the assertion $(a)$ in the theorem, we fix $\gamma\in (0,1)$, and let $\kappa=\kappa(d,\lambda,\gamma)$ be the constant from Lemma \ref{171020@lem10z}.
Denote
$$
\tilde{\omega}_{\bullet}=\tilde{\omega}_{\bullet, B_4}, \quad \cF(r)=\int_0^r \frac{\tilde{\omega}_{f_\alpha}(t)+\tilde{\omega}_{g}(t)}{t}\,dt.
$$
Similar to \eqref{171120@eq2}, we have for $x_0\in B_3$ and $0<r\le 1/4$ that
$$
|Du(x_0)-\Theta_{x_0, r}|+|p(x_0)-\theta_{x_0,r}|\lesssim \sum_{j=0}^\infty \Phi(x_0, \kappa^j r),
$$
where $\Theta_{x_0,r}\in \bR^{d\times d}$ and $\theta_{x_0,r}\in \bR$ satisfy
$$
\Phi(x_0,r)=\bigg(\dashint_{B_r(x_0)}|Du-\Theta_{x_0,r}|^q+|p-\theta_{x_0,r} |^q\,dx\bigg)^{1/q}.
$$
From this together with Lemma \ref{171020@lem10z} $(i)$, it follows that
$$
\begin{aligned}
&\big|Du(x_0)-\Theta_{x_0,r}\big|+\big|p(x_0)-\theta_{x_0,r}\big|\\
&\lesssim_{d,\lambda,\gamma} \Phi(x_0,r)+\|Du\|_{L^\infty(B_r(x_0))}\int_0^r \frac{\tilde{\omega}_{A^{\alpha\beta}}(t)}{t}\,dt+\cF(r).
\end{aligned}
$$
This implies that
$$
\begin{aligned}
|Du(x_0)|+|p(x_0)|&\lesssim r^{-d}\big(\|Du\|_{L^1(B_r(x_0))}+\|p\|_{L^1(B_r(x_0))}\big) \\
&\quad +\|Du\|_{L^\infty(B_r(x_0))}\int_0^r \frac{\tilde{\omega}_{A^{\alpha\beta}}(t)}{t}\,dt+\cF(r),
\end{aligned}
$$
which corresponds to \eqref{171024@eq5}.
Using this and following the same arguments in the proof of \eqref{170925_eq1}, we have the following $L^\infty$-estimate for $Du$ and $p$:
\begin{equation}		\label{171024@eq5b}
\|Du\|_{L^\infty(B_2)}+\|p\|_{L^\infty(B_2)} \lesssim_{d,\lambda,\gamma,\omega_{A^{\alpha\beta}}} \|Du\|_{L^1(B_6)}+\|p\|_{L^1(B_6)}+\cF(1).
\end{equation}
Similar to \eqref{171120@A2}, by Lemma \ref{171020@lem10z}, we obtain for $x_0\in B_1$ and $0<\rho \le 1/4$ that
$$
\begin{aligned}
\sum_{j=0}^\infty\Phi(x_0, \kappa^j \rho) &\lesssim_{d,\lambda.\gamma} \rho^{\gamma} \big(\|Du\|_{L^1(B_6)}+\|p\|_{L^1(B_6)}\big)\\
&\quad
+\|Du\|_{L^\infty(B_2)}\int_0^\rho\frac{\tilde{\omega}_{A^{\alpha\beta}}(t)}{t}\,dt
+\cF(\rho).
\end{aligned}
$$
Using this and \eqref{171024@eq5b}, and following the proof of \eqref{171128@A1}, we have  for any $x,y\in {B_1}$ with $|x-y|\le 1/4$ that
\begin{equation}		\label{171121@eq3}
\begin{aligned}
&|Du(x)-Du(y)|+|p(x)-p(y)|\\
&\lesssim_{d,\lambda,\gamma,\omega_{A^{\alpha\beta}}}  \big(\|Du\|_{L^1(B_6)}+\|p\|_{L^1(B_6)}\big)\left(|x-y|^\gamma+\int_0^{|x-y|} \frac{\tilde{\omega}_{A^{\alpha\beta}}(t)}{t}\,dt\right)\\
&\quad +\cF(1) \int_0^{|x-y|} \frac{\tilde{\omega}_{A^{\alpha\beta}}(t)}{t}\,dt+\cF(|x-y|).
\end{aligned}
\end{equation}
This completes the proof of the assertion $(a)$.

To prove the assertion $(b)$ in the theorem, we set $\gamma=\frac{1+\gamma_0}{2}$.
Let $\kappa=\kappa(d,\lambda,\gamma)$ be the constant from Lemma \ref{171020@lem10z}, and recall the notation
$$
\tilde{\omega}_{\bullet,B_4}(r)=\sum_{i=1}^\infty \kappa^{\gamma i} \big(\omega_{\bullet,B_4}(\kappa^{-i} r )[\kappa^{-i}r<1]+\omega_{\bullet,B_4}(1)[\kappa^{-i} r \ge 1]\big).
$$
Since it holds that (see Lemma \ref{171118@lem1})
$$
\tilde{\omega}_{f,B_4}(r)+\int_0^r \frac{\tilde{\omega}_{f,B_4}(t)}{t}\,dt \lesssim [f]_{C^{\gamma_0}(B_6)}r^{\gamma_0},
$$
by \eqref{171024@eq5b} and \eqref{171121@eq3}, we get
$$
\begin{aligned}
\|Du\|_{C^{\gamma_0}(B_1)}+\|p\|_{C^{\gamma_0}(B_1)} & \lesssim_{d,\lambda,\gamma_0,[A^{\alpha\beta}]_{C^{\gamma_0}(B_6)}} \|Du\|_{L^1(B_6)}+\|p\|_{L^1(B_6)}\\
&\quad +[f_\alpha]_{C^{\gamma_0}(B_6)}+[g]_{C^{\gamma_0}(B_6)}.
\end{aligned}
$$
The theorem is proved.
\end{proof}

\subsection{Proof of Theorem \ref{M3}}	\label{171001@sec1}
We use the idea in Brezis \cite{MR2465684} (see also \cite[Appendix]{MR2548032}).
The author proved the $W^{1,q_0}$-regularity for $W^{1,1}$-weak solutions to divergence form elliptic equations with Dini continuous coefficients by using duality and bootstrap arguments combined with regularity theories.
In this proof, we utilize our regularity result in Theorem \ref{M1} together with the $W^{1,q}$-solvability result in \cite[Theorem 2.4]{MR3758532}.

Let $\eta$ be a smooth function on $\bR^d$ satisfying
$$
0\le \eta\le 1, \quad \eta \equiv 1 \, \text{ on }\, B_{5/2}, \quad \operatorname{supp} \eta\subset B_{3}, \quad |\nabla \eta|\lesssim_d 1.
$$
Set $\tilde{\cL}$ be an operator of the form
$$
\tilde{\cL} u=D_\alpha (\tilde{A}^{\alpha\beta}D_\beta u),
$$
where $\tilde{A}^{\alpha\beta}_{ij}=\eta A^{\alpha\beta}_{ij}+\lambda(1-\eta)\delta_{\alpha\beta}\delta_{ij}$.
Here, $\lambda$ is the constant from \eqref{ell} and $\delta_{ij}$ is the Kronecker delta symbol.
Obviously, the coefficients $\tilde{A}^{\alpha\beta}$ satisfy the strong ellipticity condition \eqref{ell}.
Moreover, since $A^{\alpha\beta}$ are of partially Dini mean oscillation satisfying Definition \ref{D1} $(ii)$, $\tilde{A}^{\alpha\beta}$ are also of partially Dini mean oscillation with
\begin{align*}	
\omega_{\tilde{A}^{\alpha\beta},x'}(r):&=\sup_{x\in B_4} \dashint_{B_r(x)}\bigg|\tilde{A}^{\alpha\beta}(y)-\dashint_{B_r'(x')}\tilde{A}^{\alpha\beta}(y_1,z')\,dz'\bigg|\,dy\\
&\lesssim_{d,\lambda}  \omega_{A^{\alpha\beta},x'}(r)+r.
\end{align*}
Observe that (using Lemma \ref{171118@lem1} $(c)$)
\begin{enumerate}[$(a)$]
\item
$\Omega=B_6$ satisfies \cite[Assumption 2.1]{MR3758532} with for some $R_0\in (0,1)$.
\item
For any $\rho>0$, there exists $R_1\in (0,R_0]$, depending only on $d$, $\lambda$, $\omega_{A^{\alpha\beta},x'}$, and $\rho$, such that $\tilde{A}^{\alpha\beta}$ and $\Omega=B_6$ satisfy \cite[Assumption 2.2 ($\rho$)]{MR3758532}.
\end{enumerate}
Therefore, the $W^{1,q}$-solvability in \cite[Theorem 2.4]{MR3758532} is available for $\tilde{\cL}$ and $\tilde{\cL}^*$ on $\Omega=B_6$, where $\tilde{\cL}^*$ is the adjoint operator of $\tilde{\cL}$, i.e.,
$$
\tilde{\cL}^*v=D_\alpha(\tilde{A}^{\alpha\beta}_* D_\beta v), \quad \tilde{A}^{\alpha\beta}_*=(\tilde{A}^{\beta\alpha})^{\top}.
$$

Now let $(u,p)\in W^{1,1}(B_6)^d\times L^1(B_6)$ be a weak solution of
$$
\left\{
\begin{aligned}
\cL u+\nabla p=D_\alpha f_\alpha \quad \text{in }\, B_6,\\
\operatorname{div}u=g \quad \text{in }\, B_6,
\end{aligned}
\right.
$$
where $f_\alpha\in L^{q_0}(B_6)^d$ and $g\in L^{q_0}(B_6)$ for some $ q_0\in (1,\infty)$.
Then for any $\phi\in C^\infty_0(B_6)^d$, we have that
\begin{equation}		\label{171004@eq7}
\begin{aligned}
&\int_{B_6} \tilde{A}^{\alpha\beta}D_\beta u \cdot D_\alpha \phi\,dx= \lambda\int_{B_6}(1-\eta)D_\alpha u\cdot D_\alpha \phi\,dx\\
&\quad-\int_{B_6} A^{\alpha\beta}D_\beta u\cdot D_\alpha \eta \phi\,dx+\int_{B_6} f_\alpha \cdot D_\alpha(\eta \phi)\,dx-\int_{B_6}p \operatorname{div}(\eta \phi)\,dx.
\end{aligned}
\end{equation}
We consider the following two cases:
$$
1<q_0<\frac{d}{d-1}, \quad  \frac{d}{d-1}\le q_0<\infty.
$$
\begin{enumerate}[i.]
\item
$1<q_0<\frac{d}{d-1}$:
Set $q'_0=\frac{q_0}{q_0-1}>d$.
For $\phi_\alpha\in C^\infty_0(B_2)^d$ and $\psi\in C^\infty_0(B_2)$,
by the $W^{1,q}$-solvability in \cite[Theorem 2.4]{MR3758532},
there exists a unique $(v,\pi)\in W^{1,q_0'}_0(B_6)^d\times \tilde{L}^{q_0'}(B_6)$ satisfying
\begin{equation}		\label{171002@eq1}
\left\{
\begin{aligned}
\tilde{\cL}^* v+\nabla \pi=D_\alpha \phi_\alpha &\quad \text{in } \, B_6,\\
\operatorname{div} v=\psi-(\psi)_{B_6} &\quad \text{in }\, B_6
\end{aligned}
\right.
\end{equation}
and
\begin{equation}		\label{171122@eq1}
\|Dv\|_{L^{q_0'}(B_6)}+\|\pi\|_{L^{q_0'}(B_6)}\le C\big(\|\phi_\alpha\|_{L^{q_0'}(B_2)}+\|\psi\|_{L^{q_0'}(B_2)}\big),
\end{equation}
where $C=C(d,\lambda,\omega_{A^{\alpha\beta},x'}, q_0)$.
Since $\tilde{A}^{\alpha\beta}_*$ are of partially Dini mean oscillation, by Theorem \ref{M1} $(a)$ with scaling and covering argument, we see that $Dv$ and $\pi$ are bounded in $B_3$.
Hence, as a test function to \eqref{171002@eq1}, we can use
$\zeta u\in W^{1,1}_0(B_{5/2})^d$, where $\zeta$ is a smooth function satisfying
$$
0\le \zeta\le 1, \quad \zeta\equiv 1 \, \text{ on }\, B_2, \quad \operatorname{supp}\zeta \subset B_{5/2}, \quad |\nabla \zeta|\lesssim_d1.
$$
By testing \eqref{171002@eq1} with $\zeta u$ and setting $\phi=\zeta v$ in \eqref{171004@eq7},
we have
$$
\begin{aligned}
&\int_{B_6}D_\alpha (\zeta u) \cdot \phi_\alpha+(\zeta p) \psi \,dx=\int_{B_{5/2}} \tilde{A}^{\alpha\beta}\big(D_\beta \zeta u\cdot D_\alpha v-D_\beta u\cdot D_\alpha \zeta v\big)\,dx\\
&\quad +\int_{B_{5/2}} \pi \operatorname{div}(\zeta u)\,dx
+\int_{B_{5/2}} f_\alpha\cdot D_\alpha(\zeta v)\,dx
+\int_{B_{5/2}}p \big(\zeta (\psi)_{B_6}-\nabla \zeta\cdot v\big) \,dx.
\end{aligned}
$$
Since $\phi_\alpha$ and $\psi$ are supported in $B_2$, the left-hand side of the above inequality is equal to
$$
\int_{B_2}D_\alpha u \cdot \phi_\alpha+p \psi \,dx.
$$
Hence, by using H\"older's inequality, the Sobolev inequality, and \eqref{171122@eq1}, we obtain that
$$
\bigg|\int_{B_2}D_\alpha u \cdot \phi_\alpha+ p\psi\,dx\bigg|\le CM\big(\|\phi_\alpha\|_{L^{q_0'}(B_2)}+\|\psi\|_{L^{q_0'}(B_2)}\big),
$$
where we set
$$
M:=\|u\|_{W^{1,1}(B_{5/2})}+\|p\|_{L^1(B_{5/2})}+\|f_\alpha\|_{L^{q_0}(B_{5/2})}+\|g\|_{L^{q_0}(B_{5/2})}.
$$
Therefore, by duality and the Sobolev inequality, we get
\begin{equation}		\label{171002@eq6}
\|u\|_{W^{1,q_0}(B_2)}+\|p\|_{L^{q_0}(B_2)}\lesssim M.
\end{equation}

\item
$\frac{d}{d-1}\le q_0<\infty$:
Following the same argument used in deriving \eqref{171002@eq6}, we see that for $1\le r<R\le 2$ and $\frac{d}{d-1}\le q<\infty$, we have
\begin{equation}		\label{171002@eq5}
\begin{aligned}
&\|u\|_{W^{1,q}(B_r)}+\|p\|_{L^q(B_r)}\\
&\le C\big( \|u\|_{W^{1,q^*}(B_R)}+\|p\|_{L^{q^*}(B_R)}+\|f_\alpha\|_{L^{q}(B_R)}+\|g\|_{L^q(B_R)}\big),
\end{aligned}
\end{equation}
where $C=C(d,\lambda,\omega_{A^{\alpha\beta}},q,r,R)$ provided that $(u,p)\in W^{1,q^*}(B_R)^d\times L^{q^*}(B_R)$.
Here, $q^*$ is any number in $(1,q)$ if $q=\frac{d}{d-1}$ and $q^*=\frac{dq}{d+q}$ if $q>\frac{d}{d-1}$.

Let $k$ be the smallest positive integer such that
$$
k>\frac{dq_0-q_0-d}{q_0}.
$$
We set
$$
q_i=\frac{q_0d}{d+q_0 i}, \quad r_i=1+\frac{i}{k}, \quad i\in \{0,1,\ldots,k\}.
$$
By applying \eqref{171002@eq5} iteratively, we have
$$
\begin{aligned}
&\|u\|_{W^{1,q_0}(B_1)}+\|p\|_{L^{q_0}(B_1)} \\
&\lesssim \|u\|_{W^{1,q_k}(B_2)}+\|p\|_{L^{q_k}(B_2)}+\|f_\alpha\|_{L^{q_0}(B_2)}+\|g\|_{L^{q_0}(B_2)}.
\end{aligned}
$$
Note that $q_k<\frac{d}{d-1}$.
Therefore, using \eqref{171002@eq6} with $q_0=q_k$, we get the desired estimate.
\end{enumerate}
The theorem is proved.
\qed

\subsection{Proof of Theorem \ref{M2}}	\label{171003@sec1}
We modify the proof of Lemma \ref{170904@lem1}.
Let $F$ be a bounded linear operator on $L^1(B_6)^{d\times d}\times L^1(B_6)$ and also on $L^2(B_6)^{d\times d}\times L^2(B_6)$, and $T$ be a bounded linear operator on $L^2(B_6)^{d\times d}\times L^2(B_6)$, respectively, given as in the proof of Lemma \ref{170904@lem1}.
Then we have
$$
T_0=T\circ F,
$$
where $T_0$ is the operator mentioned in the statement of the theorem.

It suffices to show that $T$ satisfies the hypothesis of Lemma \ref{170823@lem1} with $R_0=6$, $\mu=1/3$, and $c=6$.
Fix $x_0\in B_1$ and $0<r<1/3$.
Let $\tilde{f}_\alpha\in \tilde{L}^2(B_6)^d$ and $\tilde{g}\in \tilde{L}^2(B_6)$ be supported in $B_r(x_0)$.
Assume that $(u,p)\in W^{1,2}_0(B_6)^d\times \tilde{L}^2(B_6)$ is the weak solution of \eqref{171122@eq2} with $(f_1,\ldots,f_d,g)=F^{-1}(\tilde{f}_1,\ldots,\tilde{f}_d,\tilde{g})$.
Let $R\in [6r,2)$ so that $B_1\setminus B_R(x_0)\neq \emptyset$, and let $\cL^*$ be the adjoint operator of $\cL$, i.e.,
$$
\cL^*v=D_\alpha(A^{\alpha\beta}_* D_\beta v), \quad A^{\alpha\beta}_*=(A^{\beta\alpha})^{\top}.
$$
Then by \cite[Lemma 3.2]{MR3693868}, for given
$$
\phi_\alpha\in C^\infty_0((B_{2R}(x_0)\setminus B_{R}(x_0))\cap B_1)^d \quad \text{and} \quad \psi \in C^\infty_0((B_{2R}(x_0)\setminus B_{R}(x_0))\cap B_1),
$$
there exists a unique $(v,\pi)\in W^{1,2}_0(B_6)^d\times \tilde{L}^2(B_6)$ satisfying
$$
\left\{
\begin{aligned}
\cL^* v+\nabla \pi = D_\alpha \phi_\alpha \quad \text{in }\, B_6,\\
\operatorname{div} v=\psi-(\psi)_{B_6} \quad \text{in }\, B_6,
\end{aligned}
\right.
$$
and
\begin{equation}		\label{171127@A1}
\||Dv|+|\pi|\|_{L^2(B_6)}\lesssim_{d,\lambda} \||\phi_\alpha|+|\psi|\|_{L^2((B_{2R}(x_0)\setminus B_R(x_0))\cap B_1)}.
\end{equation}
Set $V:=A^{1\beta}_* D_\beta v+\pi e_1$, and observe that $(v,\pi)$ satisfies
$$
\left\{
\begin{aligned}
\cL^* v+\nabla \pi = 0 \quad \text{in }\, B_{R}(x_0),\\
\operatorname{div} v=-(\psi)_{B_6} \quad \text{in }\, B_{R}(x_0).
\end{aligned}
\right.
$$
By a similar argument that led to \eqref{171128@A1} and \eqref{171219@eq2},
we obtain that
\begin{equation}		\label{171109@eq1}
\begin{aligned}
&|V(x)-V(y)|+|D_{x'}v(x)-D_{x'}v(y)| \\
&\lesssim R^{-d/2} \||Dv|+|\pi|\|_{L^2(B_{R}(x_0))} \left(\bigg(\frac{|x-y|}{R}\bigg)^{\gamma}+\int_0^{|x-y|}\frac{\tilde{\omega}_{A^{\alpha\beta},x'}(t)}{t}\,dt\right)\\
&\quad +|(\psi)_{B_6}|\left(\left(\frac{|x-y|}{R}\right)^\gamma+\int_0^{|x-y|}\frac{\tilde{\omega}_{A^{\alpha\beta},x'}(t)}{t}\,dt\right)
\end{aligned}
\end{equation}
for any $x,y\in B_{r}(x_0)\subset B_{R/6}(x_0)$.
Note that (see \cite[Lemma 3.4]{MR3620893})
$$
\int_0^{2r}\frac{\tilde{\omega}_{A^{\alpha\beta},x'}(t)}{t}\,dt \lesssim_{d,\lambda,C_0} \left(\ln \frac{1}{r}\right)^{-1}.
$$
Therefore, by \eqref{171109@eq1} and the fact that
$$
D_1v^1=-\sum_{i=2}^d D_i v^i-(\psi)_{B_6}, \quad
$$
we have
\begin{equation}		\label{171122@A3}
\begin{aligned}
&\big\|V-(V)_{B_r(x_0)}\big\|_{L^\infty(B_r(x_0))}+\big\|D_{x'}v-(D_{x'}v)_{B_r(x_0)}\big\|_{L^\infty(B_r(x_0))}\\
&+\big\|D_1v^1-(D_1v^1)_{B_r(x_0)}\big\|_{L^\infty(B_r(x_0))} \lesssim |(\psi)_{B_6}|\left(\left(\frac{r}{R}\right)^\gamma+\left(\ln \frac{1}{r}\right)^{-1}\right)\\
&\quad + R^{-d/2}\||Dv|+|\pi|\|_{L^2(B_{R}(x_0))} \left(\left(\frac{r}{R}\right)^{\gamma}+\left(\ln \frac{1}{r}\right)^{-1}\right).
\end{aligned}
\end{equation}
Using \eqref{171127@A1} and \eqref{171122@A3}, and following the argument used in deriving \eqref{171108@eq1}, we get
$$
\begin{aligned}
&\left|\int_{(B_{2R}(x_0)\setminus B_{R}(x_0))\cap B_1} (D_\alpha u \cdot \phi_\alpha +p \psi ) \,dx\right|\\
&\lesssim \left(\left(\frac{r}{R}\right)^{\gamma}+\left(\ln \frac{1}{r}\right)^{-1}\right)R^{-d/2}M \||\phi_\alpha|+|\psi|\|_{L^2((B_{2R}(x_0)\setminus B_{R}(x_0))\cap B_1)},
\end{aligned}
$$
where
$$
M=\int_{B_r(x_0)} (|\tilde{f}_\alpha|+|\tilde{g}|\big)\,dx.
$$
Thus by duality and H\"older's inequality, we have
$$
\int_{B_{2R}(x_0)\setminus B_{R}(x_0))\cap B_1} (|Du|+|p|)\,dx
\lesssim \left(\left(\frac{r}{R}\right)^{\gamma}+\left(\ln \frac{1}{r}\right)^{-1}\right)M.
$$
Let $N$ be the smallest positive integer such that $B_1\subset B_{2^N \cdot 3 \cdot r}(x_0)$.
By taking $R=2^i\cdot 3\cdot r$, $i\in \{1,2,\ldots,N-1\}$, and using the fact that $N-1 \lesssim \ln(1/r)$,  we obtain by the above inequality that
$$
\begin{aligned}
\int_{B_1\setminus B_{6r}(x_0)} (|Du|+|p|)\,dx
\lesssim \sum_{k=1}^{N-1} \big(2^{- k\gamma}+(\ln(1/r))^{-1}\big)M\lesssim M,
\end{aligned}
$$
which implies that the map $T$ satisfies the hypothesis of Lemma \ref{170823@lem1}.
The theorem is proved.
\qed

\part{Boundary estimates}			\label{Part2}	
This part of the paper is devoted to the boundary estimates on a half ball.

\section{Main results} 		\label{Sec5}
We denote $B_r^+(x)=B_r(x)\cap \bR^d_+$, where
$$
\bR^d_+=\{x=(x_1,x')\in \bR^d:x_1>0, x'\in \bR^{d-1}\}.
$$
We use the abbreviation $B_r^+=B_r^+(0)$.

\begin{definition}		\label{D3}
Let $f\in L^1(B_6^+)$.
\begin{enumerate}[$(i)$]
\item
We say that $f$ is of {\em{Dini mean oscillation in $B_4^+$}} if the function $\omega_f:(0,1]\to [0,\infty)$ defined by
$$
\omega_f(r):=\sup_{x\in B_4^+} \dashint_{B_r^+(x)}\bigg| f(y)-\dashint_{B_r^+(x)}f(z)\,dz\bigg|\,dy
$$
satisfies
$$
\int_0^1 \frac{\omega_f(r)}{r}\,dr<\infty.
$$
\item
We say that $f$ is of {\em{partially Dini mean oscillation with respect to $x'$ in $B_4^+$}} if the function $\omega_{f,x'} : (0,1]\to[0,\infty)$ defined by
$$
\omega_{f,x'}(r):=\sup_{x\in B_4^+} \dashint_{B_r^+(x)}\bigg|f(y)-\dashint_{B_r'(x')} f(y_1,z')\,dz'\bigg|\,dy
$$
satisfies
$$
\int_0^1 \frac{\omega_{f,x'}(r)}{r}\,dr<\infty.
$$
\end{enumerate}
\end{definition}

Now we state the main results of the second part of the paper.

\begin{theorem}		\label{MT1}
Let $q_0\in (1,\infty)$.
Assume that $(u,p)\in W^{1,q_0}(B_6^+)^d\times L^{q_0}(B_6^+)$ is a weak solution of
\begin{equation}		\label{171109@eq2}
\left\{
\begin{aligned}
\cL u+\nabla p=D_\alpha f_\alpha \quad &\text{in }\, B_6^+,\\
\operatorname{div} u=g \quad &\text{in }\, B_6^+,\\
u=0 \quad &\text{on }\, B_6\cap \partial \bR^d_+,
\end{aligned}
\right.
\end{equation}
where $f_1\in L^\infty(B_6^+)^d$, $f_\alpha\in L^{q_0}(B_6^+)^d$, $\alpha\in \{2,\ldots,d\}$, and $g\in L^\infty(B_6^+)$.
Set
$$
\hat{U}:=A^{1\beta}D_\beta u+p e_1-f_1.
$$
\begin{enumerate}[$(a)$]
\item
If $A^{\alpha\beta}$, $f_\alpha$, and $g$ are of partially Dini mean oscillation with respect to $x'$ in $B_4^+$, then we have
\begin{equation}		\label{171109@eq2a}
(u,p)\in W^{1,\infty}(B_1^+)^d\times L^\infty(B_1^+)
\end{equation}
and
$$
\hat{U}, D_\alpha u\in C(\overline{B_1^+})^d, \quad \alpha\in \{2,\ldots,d\}.
$$
\item
If it holds that $[A^{\alpha\beta}]_{C^{\gamma_0}_{x'}(B_6^+)}
+[f_\alpha]_{C^{\gamma_0}_{x'}(B_6^+)}
+[g]_{C^{\gamma_0}_{x'}(B_6^+)}<\infty$ for some  $\gamma_0\in (0,1)$,
then we have \eqref{171109@eq2a} and
$$
\hat{U}, D_\alpha u\in C^{\gamma_0}(\overline{B_1^+})^d, \quad \alpha\in \{2,\ldots,d\}.
$$
\end{enumerate}
\end{theorem}

\begin{theorem}		\label{MT2}
Let $q_0\in (1,\infty)$.
Assume that $(u,p)\in W^{1,q_0}(B_6^+)^d\times L^{q_0}(B_6^+)$ is a weak solution of \eqref{171109@eq2},
where $f_\alpha\in L^{q_0}(B_6^+)^d$ and $g\in L^{q_0}(B_6^+)$.
\begin{enumerate}[$(a)$]
\item
If $A^{\alpha\beta}$, $f_\alpha$, and $g$ are of Dini mean oscillation in $B_4^+$, then we have
$$
(u,p)\in C^1(\overline{B_1^+})^d\times C(\overline{B_1^+}).
$$
\item
If it holds that $[A^{\alpha\beta}]_{C^{\gamma_0}(B_6^+)}
+[f_\alpha]_{C^{\gamma_0}(B_6^+)}
+[g]_{C^{\gamma_0}(B_6^+)}<\infty$ for some $\gamma_0\in (0,1)$,
then we have
$$
(u,p) \in C^{1,\gamma_0}(\overline{B_1^+})^d\times C^{\gamma_0}(\overline{B_1^+}).
$$
\end{enumerate}
\end{theorem}

We also obtain the following $L^{q_0}$-estimate for $W^{1,1}$-weak solutions.
See Section \ref{180102@sec1} for the proof.

\begin{theorem}		\label{MT3}
Let $q_0\in (1,\infty)$.
Assume that $(u,p)\in W^{1,1}(B_6^+)^d\times L^1(B_6^+)$ is a weak solution of \eqref{171109@eq2}, where $f_\alpha\in L^{q_0}(B_6^+)^d$ and $g\in L^{q_0}(B_6^+)$.
If $A^{\alpha\beta}$ are of partially Dini mean oscillation with respect to $x'$ in $B_4^+$, then we have $(u,p)\in W^{1,q_0}(B_1^+)^d\times L^{q_0}(B_1^+)$ with the estimate
$$
\begin{aligned}
&\|u\|_{W^{1,q_0}(B_1^+)}+\|p\|_{L^{q_0}(B_1^+)} \\
&\le C\big( \|u\|_{W^{1,1}(B_6^+)}+\|p\|_{L^1(B_6^+)}+\|f_\alpha\|_{L^{q_0}(B_6^+)}+\|g\|_{L^{q_0}(B_6^+)}\big),
\end{aligned}
$$
where the constant $C$ depends only on $d$, $\lambda$, $\omega_{A^{\alpha\beta},x'}$, and $q_0$.
\end{theorem}

\begin{remark}		\label{171122@rmk1}
By the same reasoning as in Remarks \ref{171008@rmk1} and \ref{171218@rmk1}, we can easily extend the results in Theorems \ref{MT1} and \ref{MT2} to a solution $(u,p)\in W^{1,1}(B_6^+)^d\times L^1(B_6^+)$ of the problem
$$
\left\{
\begin{aligned}
\cL u+\nabla p=f+D_\alpha f_\alpha \quad &\text{in }\, B_6^+,\\
\operatorname{div} u=g \quad &\text{in }\, B_6^+,\\
u=0 \quad &\text{on }\, B_6\cap \partial \bR^d_+,
\end{aligned}
\right.
$$
where $f\in L^q(B_6^+)^d$ with $q>d$.
\end{remark}

Finally, we prove the following weak type-$(1,1)$ estimate.
See Section \ref{171122@sec4} for the proof.

\begin{theorem}		\label{MT4}
Define a bounded linear operator $T_0$ on $L^2(B_6^+)^{d\times d}\times L^2(B_6^+)$ by
$$
T_0(f_1,\ldots,f_d,g)=(D_1u,\ldots, D_d u,p),
$$
where $(u, p)\in W^{1,2}_0(B_6^+)^d\times \tilde{L}^2(B_6^+)$ is a unique weak solution of
\begin{equation}		\label{171122@A1}
\left\{
\begin{aligned}
\cL u +\nabla p=D_\alpha f_\alpha &\quad \text{in }\, B_6^+,\\
\operatorname{div} u=g-(g)_{B_6^+} &\quad \text{in }\, B_6^+.
\end{aligned}
\right.
\end{equation}
If $A^{\alpha\beta}$ are of partially Dini mean oscillation with respect to $x'$  in $B_4^+$ and
\begin{equation}		\label{171218@eq5}
\omega_{A^{\alpha\beta},x'}(r)\le C_0 \bigg(\ln \frac{r}{4}\bigg)^{-2}, \quad \forall r\in (0,1],
\end{equation}
then the operator $T_0$ can be extended on
$$
\big\{(f_1,\ldots,f_d,g)\in L^1(B_6^+)^{d\times d}\times L^1(B_6^+): \operatorname{supp}(f_1,\ldots,f_d,g)\subset B_1^+\big\}
$$
such that for any $t>0$, we have
$$
\big|\{x\in B_1^+: |T_0 (f_1,\ldots,f_d,g)|>t\}\big|\le \frac{C}{t}\int_{B_1^+}\big(|f_\alpha|+|g|\big)\,dx,
$$
where the constant $C$ depends only on $d$, $\lambda$, $\omega_{A^{\alpha\beta},x'}$, and $C_0$.
\end{theorem}

\section{Preliminary lemmas}		\label{Sec6}
In this section, we prove some preliminary results which will be used in the proofs of the main theorems in Section \ref{171122@sec1}.
Throughout this section, we set
$$
\cL_0 u=D_\alpha(\bar{A}^{\alpha\beta}D_\beta u),
$$
where $\bar{A}^{\alpha\beta}=\bar{A}^{\alpha\beta}(x_1)$.

The following lemma is an analog of Lemma \ref{170901@lem1}.

\begin{lemma}		\label{171109@lem1}
Let $0<r<R$ and $\ell$ be a constant.
Let $(u,p)\in W^{1,2}(B_R^+)^d\times L^2(B_R^+)$ satisfy
\begin{equation}		\label{171109@eq3}
\left\{
\begin{aligned}
\cL_0 u+\nabla p=0 &\quad \text{in }\, B_R^+,\\
\operatorname{div} u=\ell &\quad \text{in }\, B^+_R,\\
u=0 &\quad \text{on }\, B_R\cap \partial \bR^d_+.
\end{aligned}
\right.
\end{equation}
Then we have
\begin{align}
\label{171109@eq3a}
\|Du\|_{L^\infty(B_r^+)}&\lesssim_{d,\lambda} (R-r)^{-d/2}\|Du\|_{L^2(B_R^+)},\\
\label{171109@eq3b}
[U]_{C^{0,1}(B_r^+)}+[D_{x'}u]_{C^{0,1}(B_r^+)} &\lesssim_{d,\lambda} (R-r)^{-d/2-1}\|Du\|_{L^2(B_R^+)},
\end{align}
where $U:=\bar{A}^{1\beta}D_\beta u+p e_1$
and the $C^{0,1}$ semi-norm is defined as in \eqref{171218@eq1}.
If $\cL_0$ has constant coefficients, then we have
\begin{equation}			\label{171110@eq3}
[Du]_{C^{0,1}(B_r^+)}+[p]_{C^{0,1}(B_r^+)}\lesssim_{d,\lambda} (R-r)^{-d/2-1}\|Du\|_{L^2(B_R^+)}.
\end{equation}
\end{lemma}

\begin{proof}
The proofs of \eqref{171109@eq3a} and \eqref{171109@eq3b} are nearly the same as that of Lemma \ref{170901@lem1}, by using  \cite[Lemma 4.1 $(ii)$]{MR3758532} instead of \cite[Lemma 4.1 $(i)$]{MR3758532}.

To show \eqref{171110@eq3}, we note that $D_1 u^1=\ell-\sum_{i=2}^d D_i u^i$.
Using this together with \eqref{171109@eq3b}, we have
\begin{equation}		\label{171110@eq3a}
[D_1u^1]_{C^{0,1}(B_r^+)}\lesssim (R-r)^{-d/2-1}\|Du\|_{L^2(B_R^+)}.
\end{equation}
Since $\bar{A}^{\alpha\beta}$ are constant and
$$
\sum_{j=2}^d \bar{A}^{11}_{ij} D_1 u^j=U^i - \sum_{j=1}^d \sum_{\beta=2}^d \bar{A}^{1\beta}_{ij}D_\beta u^j-\bar{A}^{11}_{i1}D_1 u^1, \quad i\in \{2,\ldots,d\},
$$
by \eqref{171109@eq3b} and \eqref{171110@eq3a}, we obtain that
$$
[Du]_{C^{0,1}(B_r^+)}\lesssim (R-r)^{-d/2-1}\|Du\|_{L^2(B_R^+)}.
$$
Therefore, we get \eqref{171110@eq3} from the relation $pe_1=U-\bar{A}^{1\beta}D_\beta u$.
\end{proof}

Combining Lemmas \ref{170901@lem1} and \ref{171109@lem1}, we get the following estimates on the intersection of $\bR^d_+$ and any ball centered at $z\in \bR^d_+$.

\begin{lemma}		\label{171109@lem2}
Let $z\in \bR^d_+$, $0<r<R$, and $\ell$ be a constant.
Let $(u,p)\in W^{1,2}(B_R^+(z))^d\times L^2(B_R^+(z))$ satisfy \eqref{171109@eq3} with $B_R^+(z)$ and $B_R(z)$ in place of $B_R^+$ and $B_R$, respectively.
Then we have
\begin{align}
\label{171109@eq3c}
\|Du\|_{L^\infty(B_r^+(z))}&\lesssim_{d,\lambda} (R-r)^{-d/2}\|Du\|_{L^2(B_R^+(z))},\\
\label{171109@eq3d}
[U]_{C^{0,1}(B_r^+(z))}+[D_{x'}u]_{C^{0,1}(B_r^+(z))} &\lesssim_{d,\lambda} (R-r)^{-d/2-1}\|Du\|_{L^2(B_R^+(z))},
\end{align}
where $U:=\bar{A}^{1\beta}D_\beta u+p e_1$.
If $\cL_0$ has constant coefficients, then we have
\begin{equation}			\label{171110@eq4}
[Du]_{C^{0,1}(B_r^+(z))}+[p]_{C^{0,1}(B_r^+(z))}\lesssim_{d,\lambda} (R-r)^{-d/2-1}\|Du\|_{L^2(B_R^+(z))}.
\end{equation}
\end{lemma}

\begin{proof}
We only prove the estimate \eqref{171109@eq3c} because the proofs of \eqref{171109@eq3d} and \eqref{171110@eq4} are almost the same with obvious modifications.
Let $x_0\in B_r^+(z)$ and $\rho=(R-r)/6$.
If $B_{2\rho}(x_0)\cap \partial \bR^d_+=\emptyset$, then by \eqref{170908@eq1} we have
$$
\|Du\|_{L^\infty(B_\rho(x_0))}\lesssim \rho^{-d/2}\|Du\|_{L^2(B_{2\rho}(x_0))}\lesssim (R-r)^{-d/2}\|Du\|_{L^2(B_R^+(z))}.
$$
On the other hand, if $B_{2\rho}(x_0)\cap \partial \bR^d_+\neq \emptyset$, we set $\hat{x}_0:=(0, x_0')$.
Then we have
$$
B_{\rho}^+(x_0)\subset B_{3\rho}^+(\hat{x}_0)\subset B_{4\rho}^+(\hat{x}_0)\subset B_R^+(z).
$$
Thus by \eqref{171109@eq3a} with translating the coordinates, we have
$$
\|Du\|_{L^\infty(B_{\rho}^+(x_0))}  \lesssim \rho^{-d/2}\|Du\|_{L^2(B_{4\rho}^+(\hat{x}_0))} \lesssim (R-r)^{-d/2}\|Du\|_{L^2(B_R^+(z))}.
$$
Combining the above two inequalities, we get \eqref{171109@eq3c}.
\end{proof}

Similar to Lemma \ref{170908@lem1}, we have $L^q$-mean oscillation estimates on a half ball.

\begin{lemma}		\label{171109@lem3}
Let $0<r\le R/2$ and $\ell$ be a constant.
Let $(u,p)\in W^{1,2}(B_R^+)^d\times L^2(B_R^+)$ satisfy \eqref{171109@eq3}.
Then for any $q\in (0,1)$, we have
$$
\begin{aligned}
&\left(\dashint_{B_r^+}|U-(U)_{B_r^+}|^q+|D_{x'}u|^q\,dx\right)^{1/q}\\
&\lesssim_{d,\lambda,q}
\frac{r}{R}
\inf_{\theta\in \bR^d}\bigg(\dashint_{B_R^+} |U-\theta |^q+|D_{x'}u|^q\,dx\bigg)^{1/q},
\end{aligned}
$$
where $U:=\bar{A}^{1\beta}D_\beta u+pe_1$.
If $\cL_0$ has constant coefficients, then we have
$$
\begin{aligned}
&\left(\dashint_{B_r^+}\big|D_1u-(D_1u)_{B_r^+}\big|^q+|D_{x'}u|^q+\big|p-(p)_{B_r^+}\big|^q\,dx\right)^{1/q}\\
&\lesssim_{d,\lambda,q}
\frac{r}{R}
\inf_{\theta\in \bR^d}\bigg(\dashint_{B_R^+} |D_1u-\theta|^q+|D_{x'}u|^q\,dx\bigg)^{1/q}.
\end{aligned}
$$
\end{lemma}

\begin{proof}
Since $u\equiv 0$ on $B_R\cap \partial \bR^d_+$, we have $D_{x'}u(0)=0$.
Using this and Lemma \ref{171109@lem2}, and following the proof of Lemma \ref{170908@lem1}, one can check that the estimates in the lemma hold.
We omit the details.
\end{proof}

\begin{lemma}		\label{171110@lem1}
Let $T$ be a bounded linear operator from $L^2(B_{R_0}^+)^k$ to $L^2(B_{R_0}^+)^k$, where $R_0\ge 4$ and $k\in \{1,2,\ldots\}$.
Let $\mu<1$, $c>1$, and $C>0$ be constants.
Suppose that for any $x_0\in B_1^+$, $0<r<\mu$, and
$$
g\in \tilde{L}^2(B_{R_0}^+)^k \, \text{ with }\, \operatorname{supp} g\subset B_r^+(x_0)\cap B_1^+,
$$
we have
$$
\int_{B_1^+\setminus B_{cr}(x_0)}|Tg|\,dx\le C\int_{B_r^+(x_0)\cap B_1^+}|g|\,dx.
$$
Then the operator $T$ can be extended on
$$
\big\{ f\in L^1(B_{R_0}^+)^k:\operatorname{supp}f\subset B_1^+\big\}
$$
such that for any $t>0$, we have
$$
\big|\{x\in B_1^+:|Tf(x)|>t\}\big|\lesssim_{d,k,\mu,c,C}\frac{1}{t}\int_{B_1^+}|f|\,dx.
$$
\end{lemma}
\begin{proof}
See the proof of Lemma \ref{170823@lem1}.
\end{proof}

\begin{lemma}		\label{171110@lem2}
Let $(u,p)\in W^{1,2}_0(B_4^+)^d\times \tilde{L}^2(B_4^+)$ be the weak solution of
$$
\left\{
\begin{aligned}
\cL_0 u+\nabla p=D_\alpha f_\alpha &\quad \text{in }\, B_4^+,\\
\operatorname{div} u=g-(g)_{B_4^+} &\quad \text{in }\, B^+_4,
\end{aligned}
\right.
$$
where $f_\alpha\in L^2(B_4^+)^d$ and $g\in L^2(B_4^+)$ are supported in $B_1^+$.
Then for any $t>0$, we have
$$
\big|\{x\in B_1^+:|Du(x)|+|p(x)|>t\}\big|\lesssim_{d,\lambda} \frac{1}{t}\int_{B_1^+} |f_\alpha|+|g|\,dx.
$$	
\end{lemma}

\begin{proof}
The proof of the lemma is nearly the same as that of Lemma \ref{170904@lem1}, by using \eqref{171109@eq3d} and Lemma \ref{171110@lem1} instead of \eqref{170908@eq1a} and Lemma \ref{170823@lem1}.
We omit the details.
\end{proof}

\section{Proof of Theorems \ref{MT1}, \ref{MT2}, \ref{MT3}, and \ref{MT4}}		\label{Sec7}

\subsection{Proof of Theorem \ref{MT1}}	
\label{171122@sec1}
We shall derive a priori estimates for $(u,p)$ under the assumption that $A^{\alpha\beta}$, $f_\alpha$, and $g$ are sufficiently smooth, so that $(u, p)\in C^1(\overline{B_4^+})^d\times C(\overline{B_4^+})$.

Throughout this proof, we set $q=1/2$,
$$
\Phi(x_0,r):=\inf_{\substack{\theta\in \bR^d \\ \Theta\in \bR^{d\times (d-1)}}} \bigg(\dashint_{B_r^+(x_0)} |\hat{U}-\theta|^q+|D_{x'}u-\Theta|^q\,dx\bigg)^{1/q},
$$
and
$$
\Psi(x_0,r):=\inf_{\theta\in \bR^d} \bigg(\dashint_{B_r^+(x_0)} |\hat{U}-\theta|^q+|D_{x'}u|\,dx\bigg)^{1/q},
$$
where $\hat{U}=A^{1\beta}D_\beta u+pe_1-f_1$.
For $\gamma\in (0,1)$, $\kappa\in (0,1/2]$, $m\in \{3,4\}$, and $f\in L^1(B_6^+)$, we denote
\begin{equation}		\label{171118@eq5}
\begin{aligned}
\omega_{f,x',B_m^+}(r)&:=\sup_{x\in B_m^+}\dashint_{B_r^+(x)}\bigg|f(y)-\dashint_{B_r'(x')}f(y_1,z')\,dz'\bigg|\,dy,\\
\tilde{\omega}_{f,x',B_m^+}(r)&:=\sum_{i=1}^\infty \kappa^{\gamma i} \big(\omega_{f,x',B_m^+}(\kappa^{-i} r )[\kappa^{-i}r<1]+\omega_{f,x',B_m^+}(1)[\kappa^{-i} r \ge 1]\big),\\
\omega_{f,x',B_3^+}^\sharp(r)&:=\sup_{r\le R \le 1}\left(\frac{r}{R}\right)^{\gamma}\tilde{\omega}_{f,x',B_3^+}(R).
\end{aligned}
\end{equation}

The following lemma is analogous to Lemma \ref{171020@lem1z}.

\begin{lemma}		\label{171115@lem1}
Let $\gamma\in (0,1)$.
Under the same hypothesis of Theorem \ref{MT1} $(a)$, there exists a constant $\kappa\in (0,1/2]$ depending only on $d$, $\lambda$, and $\gamma$, such that
the following hold.
\begin{enumerate}[$(i)$]
\item
For any $x_0\in B_3\cap \partial \bR^d_+$ and $0<r\le 1/4$, we have
$$
\begin{aligned}
\sum_{j=0}^\infty \Psi(x_0,\kappa^j r)&\lesssim_{d,\lambda, \gamma} \Psi(x_0,r)+\|Du\|_{L^\infty(B_r^+(x_0))}\int_0^r \frac{\tilde{\omega}_{A^{\alpha\beta},x', B_4^+}(t)}{t}\,dt\\
&\quad +\int_0^r \frac{\tilde{\omega}_{f_\alpha,x',B_4^+}(t)+\tilde{\omega}_{g,x',B_4^+}(t)}{t}\,dt.
\end{aligned}
$$
\item
For any $x_0\in B_3\cap \partial \bR^d_+$ and $0<\rho\le r\le 1/4$, we have
$$
\begin{aligned}
\Psi(x_0, \rho) &\lesssim_{d,\lambda, \gamma} \left(\frac{\rho}{r}\right)^{\gamma} \Psi(x_0, r)+\|Du\|_{L^\infty(B_r^+(x_0))}\tilde{\omega}_{A^{\alpha\beta},x',B_3^+}(\rho)\\
&\quad +\tilde{\omega}_{f_\alpha,x',B_3^+}(\rho)+\tilde{\omega}_{g,x',B_3^+}(\rho).
\end{aligned}
$$
\item
For any $x_0\in B_3^+$ and $0<r\le 1/8$, we have
$$
\begin{aligned}
&\sum_{j=0}^\infty \Phi(x_0, \kappa^j r)\lesssim_{d,\lambda,\gamma} r^{-d}\Big(\|\hat{U}\|_{L^1(B_{3r}^+(x_0))}+\|D_{x'}u\|_{L^1(B_{3r}^+(x_0))}\Big)\\
&\quad +\|Du\|_{L^\infty(B_{3r}^+(x_0))}\int_0^r \frac{\omega^\sharp_{A^{\alpha\beta},x',B_3^+}(t)}{t}\,dt+\int_0^r \frac{\omega^\sharp_{f_\alpha,x',B_3^+}(t)+\omega^\sharp_{g,x',B_3^+}(t)}{t}\,dt.
\end{aligned}
$$
\item
For any $x_0\in B_3^+$ and $0<\rho\le r\le 1/8$, we have
\begin{equation}		\label{171115@eq1}
\begin{aligned}
\Phi(x_0, \rho)&\lesssim_{d,\lambda,\gamma} \left(\frac{\rho}{r}\right)^{\gamma} r^{-d} \Big(\|\hat{U}\|_{L^1(B_{3r}^+(x_0))}+\|D_{x'}u\|_{L^1(B_{3r}^+(x_0))}\Big)\\
&\quad +\|Du\|_{L^\infty(B_{3r}^+(x_0))} {\omega}^\sharp_{A^{\alpha\beta},x',B_3^+}(\rho)+ {\omega}^\sharp_{f_\alpha,x',B_3^+}(\rho)+ {\omega}^\sharp_{g,x',B_3^+}(\rho).
\end{aligned}
\end{equation}
\end{enumerate}
In the above, each integration is finite; see Lemma \ref{171118@lem5}.
\end{lemma}

\begin{proof}
Using Lemmas \ref{171109@lem3} and \ref{171110@lem2}, and following the proof of Lemma \ref{171020@lem1z}, one can easily check that the estimates in $(i)$ and $(ii)$ hold.
The assertion $(iii)$ is an easy consequence of the assertion $(iv)$.
Indeed, for $j\in \{0,1,2,\ldots\}$, by taking $\rho=\kappa^j r$ in \eqref{171115@eq1}, we have
$$
\begin{aligned}
&\Phi(\kappa^j r)\lesssim \kappa^{\gamma j} r^{-d}\Big(\|\hat{U}\|_{L^1(B_{3r}^+(x_0))}+\|D_{x'}u\|_{L^1(B_{3r}^+(x_0))}\Big)\\
&\quad +\|Du\|_{L^\infty(B_{3r}^+(x_0))}\omega^{\sharp}_{A^{\alpha\beta},x',B_3^+}(\kappa^j r)+\omega^\sharp_{f_\alpha,x',B_3^+}(\kappa^j r)+\omega^\sharp_{g,x',B_3^+}(\kappa^j r).
\end{aligned}
$$
Taking the summations of both sides of the above inequality with respect to $j=0,1,2,\ldots$, and using Lemma \ref{171118@lem5} $(c)$, we see that the estimate in $(iii)$ holds.

It remains to show that the assertion $(iv)$ holds.
We choose $\kappa=\kappa(d,\lambda,\gamma)\in (0,1/2]$ a sufficiently small, such that Lemma \ref{171020@lem1z} and the assertions $(i)$ and $(ii)$ in the lemma are available.
If $r/6<\rho\le r$, then \eqref{171115@eq1} follows from the definition of $\Phi$.
If $0<\rho\le r/6$, then we consider the following three cases:
$$
r\le x_{01}, \quad x_{01}\le 4\rho , \quad 4\rho \le  x_{01}\le  r.
$$
\begin{enumerate}[i.]
\item
$r\le x_{01}$:
Set $R=r/4$.
Since $B_{4R}(x_0)\subset B_6^+$,
by following the proof of  \eqref{171023@1a}, we have
$$
\begin{aligned}
\Phi(x_0,\rho)&\lesssim \left(\frac{\rho}{R}\right)^{\gamma} \Phi(x_0, R)
+\|Du\|_{L^\infty(B_R(x_0))}\tilde{\omega}_{A^{\alpha\beta},x',B_3^+}(\rho)\\
&\quad +\tilde{\omega}_{f_\alpha,x',B_3^+}(\rho)+\tilde{\omega}_{g,x',B_3^+}(\rho).
\end{aligned}
$$
Thus using the fact that
$$
\tilde{\omega}_{\bullet,x',B_3^+}\le \omega^\sharp_{\bullet,x',B_3^+}, \quad \Phi(x_0,R)\lesssim R^{-d} \Big(\|\hat{U}\|_{L^1(B_R^+(x_0))}+\|D_{x'}u\|_{L^1(B_R^+(x_0))}\Big),
$$
we get \eqref{171115@eq1}.
\item
$x_{01}\le 4\rho$:
Set $\bar{x}_0=(0,x_0')\in B_3\cap \bR^d_+$.
Since $B_\rho^+(x_0)\subset B_{5\rho}^+(\bar{x}_0)$, we have
\begin{align}	
\nonumber	
\Phi(x_0, \rho)&\lesssim \Psi(\bar{x}_0, 5\rho)\\
\nonumber	
&\lesssim \left(\frac{\rho}{r}\right)^\gamma \Psi(\bar{x}_0, r) + \|Du\|_{L^\infty(B_{r}^+(\bar{x}_0))}\tilde{\omega}_{A^{\alpha\beta},x', B_3^+}(5\rho)\\
\label{171128@eq2}
&\quad +\tilde{\omega}_{f_\alpha,x', B_3^+}(5\rho)+\tilde{\omega}_{g,x',B_3^+}(5\rho),
\end{align}
where we used the assertion $(ii)$ of the lemma in the second inequality.
Since it holds that
$$
B_r^+(\bar{x}_0)\subset B_{2r}^+(x_0), \quad \tilde{\omega}_{\bullet,x',B_3^+}(5\rho)\lesssim \omega^\sharp_{\bullet,x',B_3^+}(\rho),
$$
$$
\Psi(\bar{x}_0,r)\lesssim r^{-d}\Big(\|\hat{U}\|_{L^1(B_{2r}^+(x_0))}+\|D_{x'}u\|_{L^1(B_{2r}^+(x_0))}\Big),
$$
we get \eqref{171115@eq1} from \eqref{171128@eq2}.
\item
$4\rho\le x_{01}\le r$:
Set $R=x_{01}/4$, and observe that
$$
5R\le 5r/4 \le 1/4.
$$
Since $B_{4R}(x_0)\subset B_6^+$, by following the proof of \eqref{171023@1a}, we have
\begin{equation}		\label{171118@eq4a}
\begin{aligned}
\Phi(x_0,\rho)&\lesssim \left(\frac{\rho}{R}\right)^{\gamma} \Phi(x_0, R)
+\|Du\|_{L^\infty(B_R(x_0))}\tilde{\omega}_{A^{\alpha\beta},x',B_3^+}(\rho)\\
&\quad +\tilde{\omega}_{f_\alpha,x',B_3^+}(\rho)+\tilde{\omega}_{g,x',B_3^+}(\rho).
\end{aligned}
\end{equation}
Set $\bar{x}_0=(0,x_0')\in B_3\cap \bR^d_+$.
Then similar to \eqref{171128@eq2}, we have
\begin{align}
\nonumber		
\Phi(x_0,R)&\lesssim \Psi(\bar{x}_0,5R)\\
\nonumber		
&\lesssim \left(\frac{R}{r}\right)^{\gamma} \Psi(\bar{x}_0,5r/4)+\|Du\|_{L^\infty(B_{5r/4}^+(\bar{x}_0))}\tilde{\omega}_{A^{\alpha\beta},x',B_3^+}(5R)\\
\label{171128@eq3}
&\quad + \tilde{\omega}_{f_\alpha,x',B_3^+}(5R)+\tilde{\omega}_{g,x',B_3^+}(5R).
\end{align}
Combining \eqref{171118@eq4a} and \eqref{171128@eq3}, and using the fact that
$$
B_{5r/4}^+(\bar{x}_0)\subset B_{3r}^+(x_0), \quad \left(\frac{\rho}{R}\right)^\gamma \tilde{\omega}_{\bullet,x',B_3^+}(5R)\lesssim  \omega^\sharp_{\bullet,x',B_3^+}(\rho),
$$
$$
\Psi(\bar{x}_0,5r/4)\lesssim r^{-d}\Big(\|\hat{U}\|_{L^1(B_{3r}^+(x_0))}+\|D_{x'}u\|_{L^1(B_{3r}^+(x_0))}\Big),
$$
we get \eqref{171115@eq1}.
\end{enumerate}
The lemma is proved.
\end{proof}

Now we are ready to prove the assertion $(a)$ in the theorem.

\begin{proof}[Proof of Theorem \ref{MT1} $(a)$]
In this proof, we fix $\gamma\in (0,1)$, and let $\kappa=\kappa(d,\lambda,\gamma)$ be the constant from Lemma \ref{171115@lem1}.
We denote
$$
\omega_{\bullet,x'}=\omega_{\bullet,x',B_4^+}, \quad \tilde{\omega}_{\bullet,x'}=\tilde{\omega}_{\bullet,x',B_4^+}, \quad \omega_{\bullet,x'}^\sharp=\omega^\sharp_{\bullet,x',B_3^+}, \quad \cU=|\hat{U}|+|D_{x'}u|,
$$
$$
\cF(r)=\int_0^r \frac{\tilde{\omega}_{f_\alpha,x'}(t)+\tilde{\omega}_{g,x'}(t)}{t}\,dt, \quad \cG(r)=\int_0^r \frac{\omega^\sharp_{f_\alpha,x'}(t)+\omega^\sharp_{g,x'}(t)}{t}\,dt.
$$
By using Lemma \ref{171115@lem1} $(iii)$ and following the same steps as in the proof of \eqref{170925_eq1},
we get the $L^\infty$-estimate for $Du$ and $p$:
\begin{equation}		\label{171120@eq5}	
\begin{aligned}
&\|Du\|_{L^\infty(B_2^+)}+\|p\|_{L^\infty(B_2^+)}\\
&\le C \big(\|Du\|_{L^1(B_6^+)}+\|p\|_{L^1(B_6^+)}+\|f_1\|_{L^\infty(B_6^+)}+\|g\|_{L^\infty(B_{6}^+)}+\cG(1)\big),
\end{aligned}
\end{equation}
where $C=C(d,\lambda, \gamma, \omega_{A^{\alpha\beta},x'})$.
Similar to Remark \ref{180305@rmk1}, the parameter $\omega_{A^{\alpha\beta},x'}$ in the dependency of the constant $C$ in \eqref{171120@eq5} can be replaced by a function $\omega_0:(0,1]\to [0,\infty)$ such that
$$
\int_0^r \frac{\omega^\sharp_{A^{\alpha\beta},x'}(t)}{t}\,dt\le \int_0^r \frac{\omega_0(t)}{t}\,dt<\infty \quad \text{for all }\, r\in (0,1].
$$

To derive the estimates of the modulus of continuity of $\hat{U}$ and $D_{x'}u$,
we claim that for any $x_0\in B_1^+$ and $0<\rho\le 1/20$,
\begin{equation}		\label{171120@A1}
\begin{aligned}
&\sum_{j=0}^\infty \Phi(x_0, \kappa^j \rho) \lesssim_{d,\lambda,\gamma} \rho^{\gamma}\|\cU\|_{L^1(B_6^+)}\\
&\quad +\|Du\|_{L^\infty(B_2^+)}\int_0^{5\rho}\frac{\tilde{\omega}_{A^{\alpha\beta},x'}(t)
+\omega^\sharp_{A^{\alpha\beta},x'}(t)}{t}\,dt+\cF(5\rho)+\cG(\rho).
\end{aligned}
\end{equation}
We consider the following two cases:
$$
4\rho\le x_{01} \quad \text{and}\quad 4\rho>x_{01}.
$$
\begin{enumerate}[i.]
\item
$4\rho\le x_{01}$:

Since $B_{4\rho}(x_0)\subset B_6^+$, by following the proof of \eqref{171023@eq9}, we have
$$
\sum_{j=0}^\infty \Phi(x_0, \kappa^j \rho) \lesssim  \Phi(x_0,\rho)+\|Du\|_{L^\infty(B_\rho(x_0))}\int_0^\rho \frac{\tilde{\omega}_{A^{\alpha\beta},x'}(t)}{t}\,dt+\cF(\rho).
$$
Thus, applying Lemma \ref{171115@lem1} $(iv)$ with $r=1/8$ and Lemma \ref{171118@lem5} $(c)$, we get
\begin{equation}		\label{171216@A1}
\begin{aligned}
&\sum_{j=0}^\infty \Phi(x_0, \kappa^j \rho) \lesssim \rho^{\gamma}\|\cU\|_{L^1(B_6^+)}\\
&\quad +\|Du\|_{L^\infty(B_2^+)}\int_0^{\rho}\frac{\tilde{\omega}_{A^{\alpha\beta},x'}(t)
+\omega^\sharp_{A^{\alpha\beta},x'}(t)}{t}\,dt+\cF(\rho)+\cG(\rho),
\end{aligned}
\end{equation}
which gives \eqref{171120@A1}.
\item
$4\rho>x_{01}$:
Let $i_0$ be the integer such that $4\kappa^{i_0+1}\rho \le x_{01}<4\kappa^{i_0}\rho$.
Since $B_{4\kappa^{i_0+1}\rho}(x_0)\subset B_6^+$,
similar to \eqref{171216@A1}, we have
$$
\begin{aligned}
&\sum_{j=i_0+1}^\infty \Phi(x_0, \kappa^j\rho)=\sum_{j=0}^\infty \Phi(x_0, \kappa^{j+i_0+1}\rho)
\lesssim (\kappa^{i_0}\rho)^{\gamma}\|\cU\|_{L^1(B_6^+)}\\
&\quad +\|Du\|_{L^\infty(B_2^+)}\int_0^{\kappa^{i_0}\rho} \frac{\tilde{\omega}_{A^{\alpha\beta},x'}(t)+\omega^\sharp_{A^{\alpha\beta},x'}(t)}{t}\,dt +\cF(\kappa^{i_0}\rho)+\cG(\kappa^{i_0}\rho).
\end{aligned}
$$
Thus we get (using $\kappa^{i_0}\rho\le \rho$)
\begin{equation}		\label{171216@A2}
\begin{aligned}
&\sum_{j=i_0+1}^\infty \Phi(x_0, \kappa^j\rho)
\lesssim \rho^{\gamma}\|\cU\|_{L^1(B_6^+)}\\
&\quad +\|Du\|_{L^\infty(B_2^+)}\int_0^{\rho} \frac{\tilde{\omega}_{A^{\alpha\beta},x'}(t)+\omega^\sharp_{A^{\alpha\beta},x'}(t)}{t}\,dt +\cF(\rho)+\cG(\rho).
\end{aligned}
\end{equation}
Set $\bar{x}_0=(0,x'_0)\in B_1\cap \partial \bR^d_+$.
Since $B_{\kappa^j\rho}(x_0)\subset B_{5\kappa^{j}\rho}(\bar{x}_0)$ for $j\in \{0,1,\ldots,i_0\}$, we have
\begin{align*}
\sum_{j=0}^{i_0}\Phi(x_0,\kappa^j \rho)&\lesssim_d \sum_{j=0}^{i_0}\Psi(\bar{x}_0,5\kappa^j \rho)\\
& \lesssim \Psi (\bar{x}_0, 5\rho)+\|Du\|_{L^\infty(B_2^+)}\int_0^{5\rho}\frac{\tilde{\omega}_{A^{\alpha\beta},x'}(t)}{t}\,dt+\cF(5\rho),
\end{align*}
where we used Lemma \ref{171115@lem1} $(i)$ in the second inequality.
Thus, applying  Lemma \ref{171115@lem1} $(ii)$ with $r=1/4$ and using the fact that
$$
 \Psi(\bar{x}_0,1/4)\lesssim \|\cU\|_{L^1(B_6^+)}, \quad
\tilde{\omega}_{\bullet,x',B_3^+}(\rho)\lesssim \int_0^\rho \frac{\tilde{\omega}_{\bullet,x'}(t)}{t}\,dt,
$$
we have
$$
\sum_{j=0}^{i_0}\Phi(x_0, \kappa^j \rho)\lesssim \rho^\gamma \|\cU\|_{L^1(B_6^+)}+\|Du\|_{L^\infty(B_2^+)}\int_0^{5\rho}\frac{\tilde{\omega}_{A^{\alpha\beta},x'}(t)}{t}\,dt+\cF(5\rho).
$$
Combining this and \eqref{171216@A2}, we get the claim \eqref{171120@A1}.
\end{enumerate}

Now we are ready to obtain the estimates of the modulus of continuity of $\hat{U}$ and $D_{x'}u$.
For $x_0\in B_1^+$ and $0<\rho\le 1/20$, we denote $\theta_{x_0,\rho}\in \bR^d$ and $\Theta_{x_0, \rho}\in \bR^{d\times (d-1)}$ to be such that
$$
\Phi(x_0, \rho)=\bigg(\dashint_{B_\rho^+(x_0)}|\hat{U}-\theta_{x_0,\rho}|^q+|D_{x'}u-\Theta_{x_0, \rho}|^q\,dx\bigg)^{1/q}.
$$
Then similar to \eqref{171120@eq2}, we have
\begin{equation}		\label{171121@eq1a}
|\hat{U}(x_0)-\theta_{x_0,\rho}|+|D_{x'}u(x_0)-\Theta_{x_0,\rho}|\lesssim \sum_{j=0}^\infty \Phi(x_0, \kappa^j \rho).
\end{equation}
Let $x,y\in B_1^+$ with $\rho:=|x-y|\le 1/20$.
For any $z\in B_\rho^+(x)\cap B_\rho^+(y)$, we have
\begin{align*}
|\hat{U}(x)-\hat{U}(y)|^q &\le |\hat{U}(x)-\theta_{x,\rho}|^q+|\theta_{x,\rho}-\theta_{y,\rho}|^q+|\hat{U}(y)-\theta_{y,\rho}|^q\\
&\le 2\sup_{x_0\in B_1^+}|\hat{U}(x_0)-\theta_{x_0,\rho}|^q+|\hat{U}(z)-\theta_{x,\rho}|^q+|\hat{U}(z)-\theta_{y,\rho}|^q.
\end{align*}
By taking average over $z\in B_\rho^+(x)\cap B_\rho^+(y)$ and taking the $q$-th root, we see that
\begin{align*}
|\hat{U}(x)-\hat{U}(y)|&\lesssim \sup_{x_0\in B_1^+}|\hat{U}(x_0)-\theta_{x_0,\rho}|^q+\Phi(x,\rho)+\Phi(y,\rho)\\
&\lesssim \sup_{x_0\in B_1^+} \Bigg(\sum_{j=0}^\infty \Phi(x_0, \kappa^j \rho)+\Phi(x_0, \rho)\Bigg)\lesssim \sup_{x_0\in B_1^+}\sum_{j=0}^\infty \Phi(x_0, \kappa^j \rho),
\end{align*}
where we used \eqref{171121@eq1a} in the second inequality.
Hence we get from \eqref{171120@A1} that
$$
\begin{aligned}
&|\hat{U}(x)-\hat{U}(y)|\lesssim \rho^{\gamma}\|\cU\|_{L^1(B_6^+)}\\
&\quad +\|Du\|_{L^\infty(B_2^+)}\int_0^{5\rho}\frac{\tilde{\omega}_{A^{\alpha\beta},x'}(t)+\omega^\sharp_{A^{\alpha\beta},x'}(t)}{t}\,dt+\cF(5\rho)+\cG(\rho).
\end{aligned}
$$

Similarly, we get the same estimate for $D_{x'} u$, and thus, using \eqref{171120@eq5}, we conclude that
\begin{equation}		\label{171128@B1}
\begin{aligned}
&|\hat{U}(x)-\hat{U}(y)|+|D_{x'}u(x)-D_{x'}u(y)|\\
&\le C |x-y|^{\gamma}\big(\|Du\|_{L^1(B_6^+)}+\|p\|_{L^1(B_6^+)}+\|f_1\|_{L^1(B_6^+)}\big)\\
&\quad +C\big(\|Du\|_{L^1(B_6^+)}+\|p\|_{L^1(B_6^+)}\big)\int_0^{5|x-y|}\frac{\tilde{\omega}_{A^{\alpha\beta},x'}(t)
+\omega^\sharp_{A^{\alpha\beta},x'}(t)}{t}\,dt\\
&\quad +C\big(\|f_1\|_{L^\infty(B_6^+)}+\|g\|_{L^\infty(B_6^+)}\big)
\int_0^{5|x-y|}\frac{\tilde{\omega}_{A^{\alpha\beta},x'}(t)+\omega^\sharp_{A^{\alpha\beta},x'}(t)}{t}\,dt\\
&\quad +C\cG(1)\int_0^{5|x-y|}\frac{\tilde{\omega}_{A^{\alpha\beta},x'}(t)+\omega^\sharp_{A^{\alpha\beta},x'}(t)}{t}\,dt
+C\cF(5|x-y|)+C\cG(|x-y|)
\end{aligned}
\end{equation}
for any $x,y\in B_1^+$ with $|x-y|\le 1/20$,
where $C>0$ is a constant depending only on $d$, $\lambda$, $\gamma$, and $\omega_{A^{\alpha\beta},x'}$.
We note that
if $x,y\in B_1^+$ with $|x-y|>1/20$, then by \eqref{171120@eq5}, we have
\begin{equation}		\label{171219@eq3}
\begin{aligned}
&|\hat{U}(x)-\hat{U}(y)|+|D_{x'}u(x)-D_{x'}u(y)| \\
& \le C |x-y|^\gamma \big(\|Du\|_{L^1(B_6^+)}+\|p\|_{L^1(B_6^+)}+\|f_1\|_{L^\infty(B_6^+)}+\|g\|_{L^\infty(B_6^+)}+\cG(1)\big).
\end{aligned}
\end{equation}
The assertion $(a)$ in Theorem \ref{MT1} is proved.
\end{proof}

We now turn to the proof of the assertion $(b)$ in the theorem.

\begin{proof}[Proof of Theorem \ref{MT1} $(b)$]
In this proof, we set $\gamma=\frac{1+\gamma_0}{2}\in (0,1)$.
Let $\kappa=\kappa(d,\lambda,\gamma)$ be the constant from Lemma \ref{171115@lem1}, and recall the notation
$$
\begin{aligned}
\tilde{\omega}_{\bullet,x',B_4^+}(r)&=\sum_{i=1}^\infty \kappa^{\gamma i} \big(\omega_{\bullet,x',B_4^+}(\kappa^{-i} r )[\kappa^{-i}r<1]+\omega_{\bullet,x',B_4^+}(1)[\kappa^{-i} r \ge 1]\big),\\
\omega^\sharp_{\bullet,x',B_3^+}(r)&=\sup_{r\le R\le 1}\left(\frac{r}{R}\right)^\gamma \tilde{\omega}_{\bullet,x',B_3^+}(R).
\end{aligned}
$$
Notice from Lemma \ref{171118@lem5} $(b)$ that for any function $f$ satisfying $[f]_{C^{\gamma_0}_{x'}(B_6^+)}<\infty$, we have
$$
\tilde{\omega}_{f,x',B_4^+}(r)+\int_0^r \frac{\tilde{\omega}_{f,x',B_4^+}(t)+\omega^\sharp_{f,x',B_3^+}(t)}{t}\,dt \lesssim [f]_{C^{\gamma_0}_{x'}(B_6^+)}r^{\gamma_0}.
$$
Therefore, by \eqref{171120@eq5}, we get the following $L^\infty$-estimate for $Du$ and $p$:
$$
\begin{aligned}
&\|Du\|_{L^\infty(B_2^+)}+\|p\|_{L^\infty(B_2^+)}\lesssim_{d,\lambda,\gamma_0, [A^{\alpha\beta}]_{C^{\gamma_0}_{x'}(B_6^+)}} \|Du\|_{L^1(B_6^+)}+\|p\|_{L^1(B_6^+)}\\
&\quad +\|f_1\|_{L^\infty(B_6^+)}+\|g\|_{L^\infty(B_6^+)}+[f_\alpha]_{C^{\gamma_0}_{x'}(B_6^+)}+[g]_{C^{\gamma_0}_{x'}(B_6^+)}.
\end{aligned}
$$
Moreover, by \eqref{171128@B1} and \eqref{171219@eq3}, we obtain the following $C^{\gamma_0}$-estimate for $\hat{U}$ and $D_{x'}u$:
$$
\begin{aligned}
&[\hat{U}]_{C^{\gamma_0}(B_1^+)}+[D_{x'}u]_{C^{\gamma_0}(B_1^+)}\lesssim_{d,\lambda,\gamma_0, [A^{\alpha\beta}]_{C^{\gamma_0}_{x'}(B_6^+)}} \|Du\|_{L^1(B_6^+)}+\|p\|_{L^1(B_6^+)}\\
&\quad +\|f_1\|_{L^\infty(B_6^+)}+\|g\|_{L^\infty(B_6^+)}+[f_\alpha]_{C^{\gamma_0}_{x'}(B_6^+)}+[g]_{C^{\gamma_0}_{x'}(B_6^+)}.
\end{aligned}
$$
This completes the proof of the assertion $(b)$ in Theorem \ref{MT1} and that of Theorem \ref{MT1}.
\end{proof}

\subsection{Proof of Theorem \ref{MT2}}
\label{171122@sec2}
We shall derive a priori estimates for $(u,p)$ by assuming $(u,p)\in C^1(\overline{B_4^+})^d\times C(\overline{B_4^+})$.
We set $q=1/2$,
$$
\Phi(x_0,r):=\inf_{\substack{\theta\in \bR \\ \Theta\in \bR^{d\times d}}} \bigg(\dashint_{B_r^+(x_0)}|Du-\Theta|^q+|p-\theta|^q\,dx\bigg)^{1/q},
$$
$$
\Psi(x_0,r):=\inf_{\substack{\theta\in \bR \\ \Theta\in \bR^{d}}}\bigg(\dashint_{B_r^+(x_0)}|D_1u-\Theta|^q+|D_{x'}u|^q+|p-\theta|^q\,dx\bigg)^{1/q}.
$$
For $\gamma\in (0,1)$, $\kappa\in (0,1/2]$, $m\in \{3,4\}$, and $f\in L^1(B_6^+)$, we denote
\begin{equation}		\label{171217@eq2}
\begin{aligned}
\omega_{f,B_m^+}(r)&:=\sup_{x\in B_m^+}\dashint_{B_r^+(x)}\big|f(y)-(f)_{B_r(x)}\big|\,dy,\\
\tilde{\omega}_{f,B_m^+}(r)&:=\sum_{i=1}^\infty \kappa^{\gamma i} \big(\omega_{f,B_m^+}(\kappa^{-i} r )[\kappa^{-i}r<1]+\omega_{f,B_m^+}(1)[\kappa^{-i} r \ge 1]\big),\\
\omega_{f,B_3^+}^\sharp(r)&:=\sup_{r\le R \le 1}\left(\frac{r}{R}\right)^{\gamma}\tilde{\omega}_{f,B_3^+}(R).
\end{aligned}
\end{equation}

\begin{lemma}		\label{171121@lem1}
Let $\gamma\in (0,1)$.
Under the same hypothesis of Theorem \ref{MT2} $(a)$, there exists a constant $\kappa\in (0,1/2]$ depending only on $d$, $\lambda$, and $\gamma$, such that
the following hold.
\begin{enumerate}[$(i)$]
\item
For any $x_0\in B_3\cap \partial \bR^d_+$ and $0<r\le 1/4$, we have
$$
\begin{aligned}
\sum_{j=0}^\infty \Psi(x_0,\kappa^j r)&\lesssim_{d,\lambda, \gamma} \Psi(x_0,r)+\|Du\|_{L^\infty(B_r^+(x_0))}\int_0^r \frac{\tilde{\omega}_{A^{\alpha\beta}, B_4^+}(t)}{t}\,dt\\
&\quad +\int_0^r \frac{\tilde{\omega}_{f_\alpha,B_4^+}(t)+\tilde{\omega}_{g,B_4^+}(t)}{t}\,dt.
\end{aligned}
$$
\item
For any $x_0\in B_3\cap \partial \bR^d_+$ and $0<\rho\le r\le 1/4$, we have
$$
\begin{aligned}
\Psi(x_0, \rho) &\lesssim_{d,\lambda, \gamma} \left(\frac{\rho}{r}\right)^{\gamma} \Psi(x_0, r)+\|Du\|_{L^\infty(B_r^+(x_0))}\tilde{\omega}_{A^{\alpha\beta},B_3^+}(\rho)\\
&\quad +\tilde{\omega}_{f_\alpha,B_3^+}(\rho)+\tilde{\omega}_{g,B_3^+}(\rho).
\end{aligned}
$$
\item
For any $x_0\in B_3^+$ and $0<r\le 1/8$, we have
$$
\begin{aligned}
&\sum_{j=0}^\infty \Phi(x_0, \kappa^j r)\lesssim_{d,\lambda,\gamma} r^{-d}\Big(\|Du\|_{L^1(B_{3r}^+(x_0))}+\|p\|_{L^1(B_{3r}^+(x_0))}\Big)\\
&\quad +\|Du\|_{L^\infty(B_{3r}^+(x_0))}\int_0^r \frac{\omega^\sharp_{A^{\alpha\beta},B_3^+}(t)}{t}\,dt+\int_0^r \frac{\omega^\sharp_{f_\alpha,B_3^+}(t)+\omega^\sharp_{g,B_3^+}(t)}{t}\,dt.
\end{aligned}
$$
\item
For any $x_0\in B_3^+$ and $0<\rho\le r\le 1/8$, we have
$$
\begin{aligned}
\Phi(x_0, \rho)&\lesssim_{d,\lambda,\gamma} \left(\frac{\rho}{r}\right)^{\gamma} r^{-d} \Big(\|\hat{U}\|_{L^1(B_{3r}^+(x_0))}+\|D_{x'}u\|_{L^1(B_{3r}^+(x_0))}\Big)\\
&\quad +\|Du\|_{L^\infty(B_{3r}^+(x_0))} {\omega}^\sharp_{A^{\alpha\beta},B_3^+}(\rho)+ {\omega}^\sharp_{f_\alpha,B_3^+}(\rho)+ {\omega}^\sharp_{g,B_3^+}(\rho).
\end{aligned}
$$
\end{enumerate}
In the above, each integration is finite; see Lemma \ref{171118@lem5}.
\end{lemma}

\begin{proof}
Using Lemmas \ref{171109@lem3} and \ref{171110@lem2}, and following the proof of Lemma \ref{171020@lem10z}, one can easily check that the estimates in $(i)$ and $(ii)$ hold.
For the estimates in $(iii)$ and $(iv)$, see the proofs of Lemma \ref{171115@lem1} $(iii)$ and $(iv)$ with obvious modifications.
\end{proof}

Now we are ready to prove the theorem.

\begin{proof}[Proof of Theorem \ref{MT2}]
To prove the assertion $(a)$ in the theorem, we fix $\gamma\in (0,1)$, and let $\kappa=\kappa(d,\lambda,\gamma)$ be the constant from Lemma \ref{171121@lem1}.
We denote
$$
\omega_{\bullet}=\omega_{\bullet,B_4^+}, \quad \tilde{\omega}_{\bullet}=\tilde{\omega}_{\bullet,B_4^+}, \quad \omega_{\bullet}^\sharp=\omega^\sharp_{\bullet,B_3^+},
$$
$$
\cF(r)=\int_0^r \frac{\tilde{\omega}_{f_\alpha}(t)+\tilde{\omega}_{g}(t)}{t}\,dt, \quad \cG(r)=\int_0^r \frac{\omega^\sharp_{f_\alpha}(t)+\omega^\sharp_{g}(t)}{t}\,dt.
$$
By using Lemma \ref{171121@lem1} $(iii)$ and following the same steps as in the proof of \eqref{171024@eq5b},
we get the $L^\infty$-estimate for $Du$ and $p$:
\begin{equation}		\label{171121@eq4}	
\|Du\|_{L^\infty(B_2^+)}+\|p\|_{L^\infty(B_2^+)} \lesssim_{d,\lambda,\gamma,\omega_{A^{\alpha\beta}}} \|Du\|_{L^1(B_6^+)}+\|p\|_{L^1(B_6^+)}+\cG(1).
\end{equation}
To derive the estimates of the modulus of continuity of $Du$ and $p$,
we observe that by using Lemma \ref{171121@lem1}, and following the proof of \eqref{171120@A1}, we have
for any $x_0\in B_1^+$ and $0<\rho\le 1/20$ that
$$
\begin{aligned}
\sum_{j=0}^\infty \Phi(x_0, \kappa^j \rho) &\lesssim_{d,\lambda,\gamma} \rho^{\gamma}\big(\|Du\|_{L^1(B_6^+)}+\|p\|_{L^1(B_6^+)}\big)\\
&+\|Du\|_{L^\infty(B_2^+)}\int_0^{5\rho}\frac{\tilde{\omega}_{A^{\alpha\beta}}(t)
+\omega^\sharp_{A^{\alpha\beta}}(t)}{t}\,dt+\cF(5\rho)
+\cG(\rho).
\end{aligned}
$$
Using this and following the same steps as in the proof of \eqref{171128@B1}, we have for $x,y\in B_1^+$ with $|x-y|\le 1/20$ that
\begin{align}
\nonumber
&|Du(x)-Du(y)|+|p(x)-p(y)|\\
\nonumber
&\le C \big(\|Du\|_{L^1(B_6^+)}+\|p\|_{L^1(B_6^+)}\big)
\left(|x-y|^{\gamma}+\int_0^{5|x-y|}\frac{\tilde{\omega}_{A^{\alpha\beta}}(t)+\omega^\sharp_{A^{\alpha\beta}}(t)}{t}\,dt\right)\\
\label{171121@eq4a}
&\quad +C\cG(1) \int_0^{5|x-y|}\frac{\tilde{\omega}_{A^{\alpha\beta}}(t) +\omega^\sharp_{A^{\alpha\beta}}(t)}{t}\,dt+C\cF(5|x-y|) +C\cG(|x-y|),
\end{align}
where $C>0$ is a constant depending only on $d$, $\lambda$, $\gamma$, and $\omega_{A^{\alpha\beta}}$.
This completes the proof of the assertion $(a)$ in the theorem.

To prove the assertion $(b)$, we set $\gamma=\frac{1+\gamma_0}{2}$.
Let $\kappa=\kappa(d,\lambda,\gamma)$ be the constant from Lemma \ref{171121@lem1}.
Since it holds that (see Lemma \ref{171118@lem5})
$$
\tilde{\omega}_{f,B_4^+}(r)+\int_0^r \frac{\tilde{\omega}_{f,B_4^+}(t)+\omega^\sharp_{f,B_3^+}(t)}{t}\,dt \lesssim [f]_{C^{\gamma_0}(B_6^+)}r^{\gamma_0},
$$
by \eqref{171121@eq4} and \eqref{171121@eq4a}, we have
$$
\begin{aligned}
\|Du\|_{C^{\gamma_0}(B_1^+)}+\|p\|_{C^{\gamma_0}(B_1^+)}&\lesssim_{d,\lambda,\gamma_0, [A^{\alpha\beta}]_{C^{\gamma_0}(B_6^+)}} \|Du\|_{L^1(B_6^+)}+\|p\|_{L^1(B_6^+)}\\
&\quad +[f_\alpha]_{C^{\gamma_0}(B_6^+)}+[g]_{C^{\gamma_0}(B_6^+)}.
\end{aligned}
$$
The theorem is proved.
\end{proof}

\subsection{Proof of Theorem \ref{MT3}}
\label{180102@sec1}
The proof of the theorem is similar to that of Theorem \ref{M3}.
Let $\eta$ be a smooth function on $\bR^d$ satisfying
$$
0\le \eta\le 1, \quad \eta \equiv 1 \, \text{ on }\, B_{5/2}, \quad \operatorname{supp} \eta\subset B_{3}, \quad |\nabla \eta|\lesssim_d 1.
$$
Set $\tilde{\cL} u=D_\alpha (\tilde{A}^{\alpha\beta}D_\beta u)$,
where $\tilde{A}^{\alpha\beta}_{ij}=\eta A^{\alpha\beta}_{ij}+\lambda(1-\eta)\delta_{\alpha\beta}\delta_{ij}$.
Then $\tilde{A}^{\alpha\beta}$ are of partially Dini mean oscillation in $B_4^+$ with
$$
\begin{aligned}
\omega_{\tilde{A}^{\alpha\beta},x'}(r):&=\sup_{x\in B_4^+} \dashint_{B_r^+(x)}\bigg|\tilde{A}^{\alpha\beta}(y)-\dashint_{B_r'(x')}\tilde{A}^{\alpha\beta}(y_1,z')\,dz'\bigg|\,dy\\
&\lesssim_{d,\lambda} \omega_{A^{\alpha\beta},x'}(r)+r.
\end{aligned}
$$
Moreover, by the same reasoning as in Lemma \ref{171118@lem1} $(c)$, we see that for any $\rho>0$, there exists $k_0=k_0(d,\omega_{A^{\alpha\beta},x'}, \rho)\in (0,1)$ such that
\begin{equation}		\label{180102@eq5}
\sup_{r\in (0, k_0)}\omega_{\tilde{A}^{\alpha\beta},x'}(r)<\rho.
\end{equation}
We fix a domain $\Omega$ with a smooth boundary such that
$$
B_4^+\subset \Omega\subset B_6^+.
$$
Then by \eqref{180102@eq5}, we see that
\begin{enumerate}[$(a)$]
\item
$\Omega$ satisfies \cite[Assumption 2.1]{MR3758532} with for some $R_0=R_0(d)\in (0,1)$.
\item
For any $\rho>0$, there exists $R_1\in (0,R_0]$, depending only on $d$, $\lambda$, $\omega_{A^{\alpha\beta},x'}$, and $\rho$, such that $\tilde{A}^{\alpha\beta}$ and $\Omega$ satisfy \cite[Assumption 2.2 ($\rho$)]{MR3758532}.
\end{enumerate}
Therefore, the $W^{1,q}$-solvability in \cite[Theorem 2.4]{MR3758532} is available for $\tilde{\cL}$ and $\tilde{\cL}^*$ on $\Omega$, where $\tilde{\cL}^*$ is the adjoint operator of $\tilde{\cL}$.
Using this together with Theorem \ref{MT1}, and following the proof of Theorem \ref{M3} with obvious modifications,
we prove Theorem \ref{MT3}.
We omit the details.
\qed

\subsection{Proof of Theorem \ref{MT4}}
\label{171122@sec4}
Let $F$ be a bounded operator on
$L^1(B_6^+)^{d\times d}\times L^1(B_6^+)$ and also on $L^2(B_6^+)^{d\times d}\times L^2(B_6^+)$,
and $T$ be a bounded linear operator on $L^2(B_6^+)^{d\times d}\times L^2(B_6^+)$,
respectively, given as in the proof of  Lemma \ref{170904@lem1}.
Then we have
$$
T_0=T\circ F,
$$
where $T_0$ is the operator mentioned in the statement of the theorem.

It suffices to show that $T$ satisfies the hypothesis of Lemma \ref{171110@lem1} with $R_0=6$, $\mu=1/48$, and $c=96$.
Fix $x_0\in B_1^+$ and $0<r<1/48$.
Let $\tilde{f}_\alpha\in \tilde{L}^2(B_6^+)^d$ and $\tilde{g}\in \tilde{L}^2(B_6^+)$ be supported in $B_r^+(x_0)$.
Assume that $(u,p)\in W^{1,2}_0(B_6^+)^d\times \tilde{L}^2(B_6^+)$ is the weak solution of \eqref{171122@A1}
with $(f_1,\ldots,f_d,g)=F^{-1}(\tilde{f}_1,\ldots,\tilde{f}_d,\tilde{g})$.
Let $R\in [96 r,2)$ so that $B_1^+\setminus B_R(x_0)\neq \emptyset$, and let $\cL^*$ be the adjoint operator, i.e.,
$$
\cL^*v=D_\alpha(A^{\alpha\beta}_* D_\beta v), \quad A^{\alpha\beta}_*=(A^{\beta\alpha})^{\top}.
$$
Then by \cite[Lemma 3.2]{MR3693868}, for given
$$
\phi_\alpha\in C^\infty_c((B_{2R}(x_0)\setminus B_{R}(x_0))\cap B_1^+)^d\quad\text{and} \quad \psi \in C^\infty_c((B_{2R}(x_0)\setminus B_{R}(x_0))\cap B_1^+),
$$
there exists a unique $(v,\pi)\in W^{1,2}_0(B_6^+)^d\times \tilde{L}^2(B_6^+)$ satisfying
$$
\left\{
\begin{aligned}
\cL^* v+\nabla \pi = D_\alpha \phi_\alpha \quad \text{in }\, B_6^+,\\
\operatorname{div} v=\psi-(\psi)_{B_6^+} \quad \text{in }\, B_6^+.
\end{aligned}
\right.
$$
Set $V:=A^{1\beta}_* D_\beta v+\pi e_1$, and observe that $(v,\pi)$ satisfies
$$
\left\{
\begin{aligned}
\cL^* v+\nabla \pi = 0 \quad \text{in }\, B_{R}^+(x_0),\\
\operatorname{div} v=-(\psi)_{B_6^+} \quad \text{in }\, B_{R}^+(x_0).
\end{aligned}
\right.
$$
Now we claim that
\begin{equation}		\label{171122@A2}
\begin{aligned}
&\big\|V-(V)_{B_r^+(x_0)}\big\|_{L^\infty(B_r^+(x_0))}+\big\|D_{x'}v-(D_{x'}v)_{B_r^+(x_0)}\big\|_{L^\infty(B_r^+(x_0))}\\
&+\big\|D_1 v^1-(D_1v^1)_{B_r^+(x_0)}\big\|_{L^\infty(B_r^+(x_0))} \lesssim \big|(\psi)_{B_6^+}\big| \left(\left(\frac{r}{R}\right)^\gamma+\left(\ln \frac{1}{r}\right)^{-1}\right)\\
& \quad +R^{-d/2}\||Dv|+|\pi|\|_{L^2(B_{R}^+(x_0))} \left(\left(\frac{r}{R}\right)^{\gamma}+\left(\ln \frac{1}{r}\right)^{-1}\right).
\end{aligned}
\end{equation}
We consider the following two cases:
$$
x_{01} \ge R/16, \quad x_{01} < R/16.
$$
\begin{enumerate}[i.]
\item
$x_{01}\ge R/16$: In this case, we have $B_{R/16}(x_0)\subset \bR^d_+$.
Then $(v,\pi)$ satisfies
$$
\left\{
\begin{aligned}
\cL^* v+\nabla \pi = 0 \quad \text{in }\, B_{R/16}(x_0),\\
\operatorname{div} v=-(\psi)_{B_6^+} \quad \text{in }\, B_{R/16}(x_0).
\end{aligned}
\right.
$$
Therefore, by the same reasoning as in \eqref{171122@A3}, we get \eqref{171122@A2}.
\item
$x_{01}<R/16$:
Set $\bar{x}_0=(0, x_0')\in B_1\cap \partial \bR^d_+$, and observe that
$$
B_r^+(x_0)\subset B_{R/12}^+(\bar{x}_0)\subset B_{R/2}^+(\bar{x}_0)\subset B_R^+(x_0).
$$
Since $(v,\pi)$ satisfies
$$
\left\{
\begin{aligned}
\cL^* v+\nabla \pi = 0 \quad \text{in }\, B_{R/2}^+(\bar{x}_0),\\
\operatorname{div} v=-(\psi)_{B_6^+} \quad \text{in }\, B_{R/2}^+(\bar{x}_0),
\end{aligned}
\right.
$$
by a similar argument that led to \eqref{171128@B1} and \eqref{171219@eq3},
we obtain that
\begin{equation}		\label{171219@eq4a}
\begin{aligned}
&|V(x)-V(y)|+|D_{x'}v(x)-D_{x'}v(y)|\\
&\lesssim R^{-d/2} \||Dv|+|\pi|\|_{L^2(B_{R/2}^+(\bar{x}_0))}\\
&\quad \times \left(\left(\frac{|x-y|}{R}\right)^\gamma+ \int_0^{5|x-y|}\frac{\tilde{\omega}_{A^{\alpha\beta},x'}(t)+\omega^\sharp_{A^{\alpha\beta},x'}(t)}{t}\,dt\right)\\
&\quad +\big|(\psi)_{B_6^+}\big|\left(\left(\frac{|x-y|}{R}\right)^\gamma+\int_0^{5|x-y|}\frac{\tilde{\omega}_{A^{\alpha\beta},x'}(t)+\omega^\sharp_{A^{\alpha\beta},x'}(t)}{t}\,dt\right)
\end{aligned}
\end{equation}
for any $x,y\in B_r^+(x_0)$.
Note that by \eqref{171218@eq5} and \cite[Eq. (3.5)]{MR3620893}, we have
$$
\tilde{\omega}_{A^{\alpha\beta},x'}(t)\lesssim_{d,\lambda,C_0} \bigg(\ln \frac{t}{4}\bigg)^{-2}, \quad \forall t\in (0,1].
$$
Using this, it is routine to check that
\begin{equation}		\label{171219@eq4}
\int_0^{10r}\frac{\tilde{\omega}_{A^{\alpha\beta},x'}(t)+\omega^\sharp_{A^{\alpha\beta},x'}(t)}{t}\,dt\lesssim \bigg(\ln \frac{1}{r}\bigg)^{-1}.
\end{equation}
Combining \eqref{171219@eq4a} and \eqref{171219@eq4}, and utilizing the fact that
$$
D_1v^1=-\sum_{i=2}^d D_i v^i-(\psi)_{B_6^+},
$$
we get \eqref{171122@A2}.

\end{enumerate}
The estimate \eqref{171122@A2} corresponds to \eqref{171122@A3}.
The rest of the proof is identical to that of Theorem \ref{M2} and omitted.
\qed

\section{Appendix}		\label{Sec8}

In Appendix, we provide the proofs of some technical lemmas used in the previous sections.

\begin{lemma}		\label{171118@lem1}
Let $f\in L^1(B_6)$, and recall the notation \eqref{171118@eq1}.
Assume that
$$
\int_0^1 \frac{{\omega}_{f,x',B_4}(t)}{t}\,dt<\infty.
$$
\begin{enumerate}[$(a)$]
\item
For any $0<r \le 1$, we have
\begin{equation}	\label{171118@eq1a}
\sum_{j=0}^\infty \tilde{\omega}_{f,x',B_3}(\kappa^j r) \lesssim_{d,\gamma,\kappa} \int_0^r \frac{\tilde{\omega}_{f,x',B_4}(t)}{t}\, dt <\infty.
\end{equation}
\item
Let $0<\gamma_0<\gamma$ and $[f]_{C^{\gamma_0}_{x'}(B_6)}<\infty$.
Then for any $0<r\le1$, we have
$$
\tilde{\omega}_{f,x',B_4}(r)\lesssim_{\gamma_0,\gamma,\kappa}[f]_{C^{\gamma_0}_{x'}(B_6)}r^{\gamma_0}.
$$
\item
For any $\rho>0$, there exists $k_0\in (0,1)$, depending only on $d$, $\omega_{f,x',B_4}$, and $\rho$, such that
$$
\sup_{r\in (0, k_0)}\omega_{f,x',B_4}(r)<\rho.
$$
\end{enumerate}
The same results hold if the notation \eqref{171118@eq1} and $[f]_{C^{\gamma_0}_{x'}(B_6)}$ are replaced by the notation \eqref{171217@eq1} and $[f]_{C^{\gamma_0}(B_6)}$.
\end{lemma}

\begin{proof}
To prove the assertion $(a)$, we observe that
(see \cite[Lemma 1]{MR2927619})
$$
\int_0^1 \frac{\tilde{\omega}_{f,x',B_4}(t)}{t}\,dt<\infty.
$$
Hence it suffices to show that the first inequality in \eqref{171118@eq1a} holds.
For this, we claim that for any $0<r \le 1$, we have
\begin{equation}		\label{171118@eq1b}
\omega_{f,x',B_3}(r)\lesssim_{d,\kappa} \inf_{\rho\in [\kappa r, r]}\omega_{f,x',B_4}(\rho).
\end{equation}
Let $x=(x_1,x')\in B_3$ and $0< r \le 1$.
Choose numbers $a_1,\ldots,a_m$, where $m=m(\kappa)$, in the open interval $(x_1-r, x_1+r)$ such that
$$
(x_1-r,x_1+r)\subset \bigcup_{i=1}^m \left(a_i-\frac{\kappa r}{4}, a_i+\frac{\kappa r}{4}\right).
$$
We also choose points $x'_1, \ldots,x'_n$, where $n=n(d,\kappa)$, in $B_r'(x')$ satisfying
$$
B'_r(x') \subset \bigcup_{i=1}^n B_{\kappa r/4}'(x'_i).
$$
Then for any $\rho\in [\kappa r,r]$, we have
\begin{align}
\nonumber
&\dashint_{B_{r}(x)}\bigg|f(y)-\dashint_{B'_{r}(x')}f(y_1,z')\,dz'\bigg|\,dy \\
\nonumber
&\lesssim_{d,\kappa} \sum_{i=1}^m \sum_{j=1}^n \dashint_{\alpha_i+\rho/4}^{\,\,\,\alpha_i-\rho/4} \dashint_{B'_{\rho/4}(x'_j)}\bigg|f(y_1,y')-\dashint_{B'_{r}(x')}f(y_1,z')\,dz'\bigg|\,dy'\,dy_1\\
\nonumber
&\lesssim \sum_{i=1}^m \sum_{j=1}^n \dashint^{\,\,\,\alpha_i-\rho/4}_{\alpha_i+\rho/4}\dashint_{B'_{\rho/4}(x'_j)}\bigg|f(y_1,y')-\dashint_{B'_{\rho/4}(x_j')}f(y_1,z')\,dz'\bigg|\,dy'\,dy_1\\
\nonumber
&\quad + \sum_{i=1}^m \sum_{j=1}^n \dashint^{\,\,\,\alpha_i-\rho/4}_{\alpha_i+\rho/4} \bigg|\dashint_{B'_{r}(x')}f(y_1,z')\,dz'-\dashint_{B'_{\rho/4}(x'_j)}f(y_1,z')\,dz'\bigg|\,dy_1 \\
\label{170921@eq2}
&=:I_1+I_2.
\end{align}
Observe that $I_1\lesssim \omega_{f,x',B_4}(\rho)$ and
$$
I_2\lesssim \sum_{i=1}^m  \sum_{j,k=1}^n \dashint^{\,\,\,\alpha_i-\rho/4}_{\alpha_i+\rho/4} \dashint_{B'_{\rho/4}(x'_k)}\bigg|f(y_1,y')-\dashint_{B'_{\rho/4}(x'_j)}f(y_1,z')\,dz'\bigg|\,dy'\,dy_1.
$$
For given $j$ and $k$, we find $\{z'_1=x'_j,z'_2, \ldots, z'_{l_{jk}}=x'_k\}\subset \{x_1',\ldots,x'_n\}$ such that $B'_{\rho/4}(z_l')\cap B'_{\rho/4}(z_{l+1}')\neq \emptyset$.
Then we have
\begin{align*}
&\dashint_{B'_{\rho/4}(x'_k)}\bigg|f(y_1,y')-\dashint_{B'_{\rho/4}(x'_j)}f(y_1,z')\,dz'\bigg|\,dy' \\
&\lesssim \sum_{l=1}^{l_{jk}} \dashint_{B'_{\rho/2}(z_l')}\bigg|f(y_1,y')-\dashint_{B'_{\rho/2}(z_l')} f(y_1,z')\,dz'\bigg|\,dy',
\end{align*}
and thus, we get
\begin{align*}
I_2&\lesssim \sum_{i=1}^m\sum_{j,k=1}^n\sum_{l=1}^{l_{jk}}\dashint^{\,\,\,\alpha_i-\rho/4}_{\alpha_i+\rho/4}\dashint_{B'_{\rho/2}(z_l')}\bigg|f(y)-\dashint_{B'_{\rho/2}(z_l')} f(y_1,z')\,dz'\bigg|\,dy\\
&\lesssim \omega_{f,x',B_4}(\rho).
\end{align*}
From \eqref{170921@eq2} and the estimates of $I_1$ and $I_2$, the claim \eqref{171118@eq1b} follows.

Now we are ready to prove the first inequality in \eqref{171118@eq1a}.
We set
$$
\hat{\omega}_{f,x',B_m}(r)=
\left\{
\begin{aligned}
\omega_{f,x',B_m}(r) &\quad \text{if }\, r < 1,\\
\omega_{f,x',B_m}(1) &\quad \text{if }\, r \ge 1,
\end{aligned}
\right.
$$
and observe that
$$
\tilde{\omega}_{f,x',B_m}(r)=\sum_{i=1}^\infty \kappa^{\gamma i} \hat{\omega}_{f,x',B_m}(\kappa^{-i}r).
$$
By \eqref{171118@eq1b}, we have for $0<r\le 1$ that
\begin{equation}		\label{170922@eq2}
\tilde{\omega}_{f,x',B_3}(r)\lesssim_{d,\gamma,\kappa} \inf_{\rho\in [\kappa r,r]} \tilde{\omega}_{f,x',B_4}(\rho).
\end{equation}
Indeed,  for any $\rho\in [\kappa r,r]$ and $i\in \{1,2,\ldots\}$, we have
$$
\hat{\omega}_{f,x',B_3}(\kappa^{-i}r)=\omega_{f,x',B_3}(\kappa^{-i}r) \lesssim \omega_{f,x',B_4}(\kappa^{-i}\rho)=\hat{\omega}_{f,x',B_4}(\kappa^{-i}\rho), \quad \kappa^{-i} r<1,
$$
and
$$
\hat{\omega}_{f,x',B_3}(\kappa^{-i}r) = \omega_{f,x',B_3}(1)\le \omega_{f,x',B_4}(1)=\hat{\omega}_{f,x',B_4}(\kappa^{-(i+1)}\rho), \quad \kappa^{-i}r\ge1.
$$
Thus, we get
\begin{align*}
\tilde{\omega}_{f,x',B_3}(r)
&=\sum_{\kappa^{-i}r< 1}\kappa^{\gamma i} \hat{\omega}_{f,x',B_3}(\kappa^{-i}r)+\sum_{\kappa^{-i}r\ge 1} \kappa^{\gamma i} \hat{\omega}_{f,x',B_3}(\kappa^{-i}r)\\
&\lesssim \sum_{i=1}^\infty \kappa^{\gamma i} \hat{\omega}_{f,x',B_4}(\kappa^{-i}\rho)+\sum_{i=1}^\infty \kappa^{\gamma(i+1)}\hat{\omega}_{f,x',B_4}(\kappa^{-(i+1)}\rho)\\
&\lesssim \tilde{\omega}_{f,x',B_4}(\rho),
\end{align*}
which gives \eqref{170922@eq2}.
Therefore, we obtain by \eqref{170922@eq2} that
$$
\tilde{\omega}_{f,x',B_3}(\kappa^j r) \lesssim \int_{\kappa^{j+1} r}^{\kappa^j r} \frac{\tilde{\omega}_{f,x',B_4}(t)}{t}\,dt, \quad j\in \{0,1,2,\ldots\}.
$$
Taking the summations of both sides of the above inequality to $j=0,1,2,\ldots$,
we see that the first inequality in \eqref{171118@eq1a} holds.

To prove the assertion $(b)$, we observe that
$$
\tilde{\omega}_{f,x',B_4}(r)=\sum_{\kappa^{-i}r<1} \kappa^{\gamma i} \omega_{f,x',B_4}(\kappa^{-i}r)+\sum_{\kappa^{-i}r\ge 1}\kappa^{\gamma i} \omega_{f,x',B_4}(1).
$$
Using the fact that $\omega_{f,x',B_4}(r)\le [f]_{C^{\gamma_0}_{x'}(B_6)}r^{\gamma_0}$ for $0<r\le 1$, we have
$$
\begin{aligned}
\tilde{\omega}_{f,x',B_4}(r)
& \lesssim [f]_{C^{\gamma_0}_{x'}(B_6)} r^{\gamma_0} \sum_{i=1}^\infty (\kappa^{\gamma-\gamma_0})^{i}+[f]_{C^{\gamma_0}_{x'}(B_6)}\sum_{i\ge \ln r/\ln \kappa}  (\kappa^{\gamma})^i\\
&\lesssim [f]_{C^{\gamma_0}_{x'}(B_6)}r^{\gamma_0}.
\end{aligned}
$$

Now we turn to the proof of the assertion $(c)$.
Observe that
$$
\omega_{f,x',B_4}(r)\le C_0 \inf_{t\in [r,2r]} \omega_{f,x',B_4}(t)\le C_0 \int_0^{2r} \frac{\omega_{f,x',B_4}(t)}{t}\,dt<\infty
$$
for any $0<r\le 1/2$, where $C_0=C_0(d)$.
Thus, for given $\rho>0$, by taking $k_0\in (0,1/2]$ such that
$$
\int_0^{2k_0}\frac{\omega_{f,x',B_4}(t)}{t}\,d t<\frac{\rho}{C_0},
$$
we get the desired inequality in the assertion $(c)$.
The lemma is proved.
\end{proof}

\begin{lemma}		\label{171118@lem5}
Let $f\in L^1(B_6^+)$, and recall the notation \eqref{171118@eq5}.
Assume that
$$
\int_0^1 \frac{\omega_{f,x',B_4^+}(t)}{t}\,dt<\infty.
$$
\begin{enumerate}[$(a)$]
\item
For any $0<r\le 1$, we have
$$
\sum_{j=0}^\infty \tilde{\omega}_{f,x',B_3^+}(\kappa^j r)\lesssim_{\gamma,\kappa} \int_0^r \frac{\tilde{\omega}_{f,x',B_4^+}(t)}{t}\,dt<\infty.
$$
\item
Let $0<\gamma_0<\gamma$ and $[f]_{C^{\gamma_0}_{x'}(B_6^+)}<\infty$.
Then for any $0<r\le 1$, we have
$$
\tilde{\omega}_{f,x',B_4^+}(r)\lesssim_{\gamma_0, \gamma, \kappa}[f]_{C^{\gamma_0}_{x'}(B_6^+)}r^{\gamma_0}.
$$

\item
For any $0<r \le 1$, we have
$$
\sum_{j=0}^\infty {\omega}_{f,x',B_3^+}^\sharp(\kappa^j r)\lesssim_{\gamma,\kappa} \int_0^r \frac{\omega^\sharp_{f,x',B_3^+}(t)}{t}\,dt<\infty.
$$
\end{enumerate}
The same results hold if the notation \eqref{171118@eq5} and $[f]_{C^{\gamma_0}_{x'}(B_6^+)}$ are replaced by the notation \eqref{171217@eq2} and $[f]_{C^{\gamma_0}(B_6^+)}$.
\end{lemma}

\begin{proof}
We prove only the assertion $(c)$ because
the proofs of the assertions $(a)$ and $(b)$ are nearly the same as those of Lemma \ref{171118@lem1} $(a)$ and $(b)$.
Similar to \eqref{170922@eq2}, we have for $0<r\le 1$ that
\begin{equation}		\label{171118@eq6}
\tilde{\omega}_{f,x',B_3^+}(r)\lesssim_{d,\gamma,\kappa}\inf_{\rho\in [\kappa r,r]}\tilde{\omega}_{f,x',B_4^+}(\rho).
\end{equation}
For any $R\in [r,1]$, we choose an integer $i$ such that $\kappa^{-i}r\le R < \kappa^{-i-1}r$.
Then by \eqref{171118@eq6}, we have
$$
\left(\frac{r}{R}\right)^\gamma \tilde{\omega}_{f,x',B_3^+}(R) \lesssim_{d,\gamma,\kappa} \kappa^{\gamma i} \tilde{\omega}_{f,x',B_4^+}(\kappa^{-i}r),
$$
which implies that
$$
\omega^\sharp_{f,x',B_3^+}(r)\lesssim \sum_{i=0}^\infty \kappa^{\gamma i} \big(\tilde{\omega}_{f,x',B_4^+}(\kappa^{-i} r)[\kappa^{-i}r<1]+\tilde{\omega}_{f,x',B_4^+}(1)[\kappa^{-i}r\ge 1 ]\big).
$$
Thus using the fact that
$$
\int_0^1 \frac{\tilde{\omega}_{f,x',B_4^+}(t)}{t}\,dt<\infty,
$$
we get (see \cite[Lemma 1]{MR2927619})
$$
\int_0^{1} \frac{\omega^\sharp_{f,x',B_3^+}(t)}{t}\,dt<\infty.
$$
By the definition of $\omega^\sharp$, we have
$$
\omega^\sharp_{f,x',B_3^+}(r)\lesssim_{\gamma,\kappa}\inf_{\rho\in [\kappa r,r]}\omega^\sharp_{f,x',B_3^+}(\rho), \quad \forall r\in (0,1].
$$
Using the above two inequalities, we obtain that
$$
\sum_{j=0}^\infty \omega^\sharp_{f,x',B_3^+}(\kappa^j r)\lesssim \int_0^r \frac{\omega^\sharp_{f,x',B_3^+}(t)}{t}\,dt<\infty.
$$
This completes the proof of the assertion $(c)$.
\end{proof}

\section*{Acknowledgement}
The authors would like to thank the referee for a very careful reading of the manuscript and many useful comments.

\bibliographystyle{plain}

\end{document}